\documentclass[twoside]{amsart}

\usepackage{amsmath,amsfonts,amsthm,mathrsfs}
\usepackage{amssymb}

\usepackage[unicode,bookmarks]{hyperref} %ocgcolorlinks
\usepackage[usenames,dvipsnames]{xcolor}
\hypersetup{colorlinks=false}%,citecolor=NavyBlue,linkcolor=BrickRed,urlcolor=Green} %Orange

\usepackage{expl3}

\ExplSyntaxOn
\prop_new:N \g_cite_map_prop
\tl_new:N \l_citekey_result_tl

\cs_new:Npn \mapcitekey #1#2 {
  \clist_map_inline:nn {#2}
       {  \prop_gput:Nnn  \g_cite_map_prop  {##1} {#1}   }
}

\cs_new:Npn \getcitekey #1 {
   \prop_get:NoN \g_cite_map_prop{#1}  \l_citekey_result_tl
   \quark_if_no_value:NF \l_citekey_result_tl
       {  \tl_set_eq:NN #1  \l_citekey_result_tl  }
}

\cs_new:Npn \showcitekeymaps {\prop_show:N  \g_cite_map_prop }
\ExplSyntaxOff

\usepackage{etoolbox}
\makeatletter
\patchcmd{\@citex}{\if@filesw}{\getcitekey\@citeb \if@filesw}%
    {\typeout{*** SUCCESS ***}}{\typeout{*** FAIL ***}}
\patchcmd{\nocite}{\if@filesw}{\getcitekey\@citeb \if@filesw}%
    {\typeout{*** SUCCESS ***}}{\typeout{*** FAIL ***}}
\makeatother

\usepackage{enumitem}
\newenvironment{renumerate}{%
	\begin{enumerate}[label=(\roman{*}), ref=(\roman{*})]
}{%
	\end{enumerate}%
}

\newenvironment{aenumerate}{%
	\begin{enumerate}[label=(\alph{*}), ref=(\alph{*})]
}{%
	\end{enumerate}%
}

\usepackage{chngcntr}

\usepackage{tikz}
\usepackage{tikz-cd}
\newcommand{\Dmod}{\mathscr{D}}
\newcommand{\Mmod}{\mathcal{M}}
\newcommand{\Nmod}{\mathcal{N}}

\newcommand{\derR}{\mathbf{R}}

\newcommand{\decal}[1]{\lbrack #1 \rbrack}

\newcommand{\shH}{\mathcal{H}}

\newcommand{\tensor}{\otimes}

\newcommand{\shHom}{\mathcal{H}\hspace{-1pt}\mathit{om}}

\newcommand{\ZZ}{\mathbb{Z}}
\newcommand{\QQ}{\mathbb{Q}}
\newcommand{\RR}{\mathbb{R}}
\newcommand{\CC}{\mathbb{C}}
\newcommand{\HH}{\mathbb{H}}
\newcommand{\PP}{\mathbb{P}}

\newcommand{\menge}[2]{\bigl\{ \thinspace #1 \thinspace\thinspace \big\vert%
\thinspace\thinspace #2 \thinspace \bigr\}}
\newcommand{\Menge}[2]{\Bigl\{ \thinspace #1 \thinspace\thinspace \Big\vert%
\thinspace\thinspace #2 \thinspace \Bigr\}}

\DeclareMathOperator{\id}{id}

\renewcommand{\Im}{\operatorname{Im}}
\renewcommand{\Re}{\operatorname{Re}}

\DeclareMathOperator{\Supp}{Supp}
\DeclareMathOperator{\codim}{codim}

\DeclareMathOperator{\gr}{gr}
\DeclareMathOperator{\DR}{DR}

\DeclareMathOperator{\End}{End}
\DeclareMathOperator{\Hom}{Hom}

\DeclareMathOperator{\GL}{GL}

\DeclareMathOperator{\Pic}{Pic}

\newcommand{\define}[1]{\emph{#1}}

\newcommand{\shf}[1]{\mathscr{#1}}
\newcommand{\OX}{\shf{O}_X}
\newcommand{\OmX}{\Omega_X}

\def\overbar#1#2#3{{%
	\setbox0=\hbox{\displaystyle{#1}}%
	\dimen0=\wd0
	\advance\dimen0 by -#2 
	\vbox {\nointerlineskip \moveright #3 \vbox{\hrule height 0.3pt width \dimen0}%
		\nointerlineskip \vskip 1.5pt \box0}%
}}

\newcommand{\into}{\hookrightarrow}

\newcommand{\HR}{H_{\RR}}

\newcommand{\HC}{H_{\CC}}

\newcommand{\fu}{f^{\ast}}
\newcommand{\fl}{f_{\ast}}

\newcommand{\il}{i_{\ast}}

\newcommand{\qu}{q^{\ast}}

\newcommand{\tl}{t_{\ast}}

\newcommand{\ql}{q_{\ast}}

\newcommand{\fp}{f_{+}}

\newcommand{\DD}{\mathbb{D}}

\newcommand{\shF}{\shf{F}}
\newcommand{\shG}{\shf{G}}
\newcommand{\shE}{\shf{E}}

\newcommand{\shO}{\shf{O}}

\makeatletter
\let\@@seccntformat\@seccntformat
\renewcommand*{\@seccntformat}[1]{%
  \expandafter\ifx\csname @seccntformat@#1\endcsname\relax
    \expandafter\@@seccntformat
  \else
    \expandafter
      \csname @seccntformat@#1\expandafter\endcsname
  \fi
    {#1}%
}
\newcommand*{\@seccntformat@subsection}[1]{%
  \textbf{\csname the#1\endcsname.}
}
\makeatother

\makeatletter
\let\@paragraph\paragraph
\renewcommand*{\paragraph}[1]{%
	\vspace{0.3\baselineskip}%
	\@paragraph{\textit{#1}}%
}
\makeatother

\counterwithin{equation}{subsection}
\counterwithout{subsection}{section}
\counterwithin{figure}{subsection}

\newtheorem{theorem}[equation]{Theorem}
\newtheorem*{theorem*}{Theorem}
\newtheorem{lemma}[equation]{Lemma}
\newtheorem*{lemma*}{Lemma}
\newtheorem{corollary}[equation]{Corollary}
\newtheorem*{corollary*}{Corollary}
\newtheorem{proposition}[equation]{Proposition}
\newtheorem*{proposition*}{Proposition}

\newtheorem*{conjecture*}{Conjecture}

\theoremstyle{definition}
\newtheorem{definition}[equation]{Definition}
\newtheorem*{definition*}{Definition}
\theoremstyle{remark}
\newtheorem*{remark}{Remark}

\newtheorem*{problem}{Problem}

\newtheorem{example}[equation]{Example}
\newtheorem*{example*}{Example}
\newtheorem*{problem*}{Problem}
\newtheorem*{note}{Note}

\theoremstyle{plain}

\newcommand{\theoremref}[1]{\hyperref[#1]{Theorem~\ref*{#1}}}
\newcommand{\lemmaref}[1]{\hyperref[#1]{Lemma~\ref*{#1}}}
\newcommand{\definitionref}[1]{\hyperref[#1]{Definition~\ref*{#1}}}
\newcommand{\propositionref}[1]{\hyperref[#1]{Proposition~\ref*{#1}}}
\newcommand{\conjectureref}[1]{\hyperref[#1]{Conjecture~\ref*{#1}}}
\newcommand{\corollaryref}[1]{\hyperref[#1]{Corollary~\ref*{#1}}}
\newcommand{\exampleref}[1]{\hyperref[#1]{Example~\ref*{#1}}}
\newcommand{\exerciseref}[1]{\hyperref[#1]{Exercise~\ref*{#1}}}

\makeatletter
\let\old@caption\caption
\renewcommand*{\caption}[1]{%
	\setcounter{figure}{\value{equation}}%
	\stepcounter{equation}%
	\old@caption{#1}\relax%
}
\makeatother

\newcounter{intro}

\newtheorem{intro-conjecture}[intro]{Conjecture}
\newtheorem{intro-corollary}[intro]{Corollary}
\newtheorem{intro-theorem}[intro]{Theorem}

\newcommand{\OA}{\mathscr{O}_A}

\newcommand{\OY}{\shO_Y}

\newcommand{\parref}[1]{\hyperref[#1]{\S\ref*{#1}}}
\newcommand{\chapref}[1]{\hyperref[#1]{Chapter~\ref*{#1}}}

\makeatletter
\newcommand*\if@single[3]{%
  \setbox0\hbox{${\mathaccent"0362{#1}}^H$}%
  \setbox2\hbox{${\mathaccent"0362{\kern0pt#1}}^H$}%
  \ifdim\ht0=\ht2 #3\else #2\fi
  }
\newcommand*\rel@kern[1]{\kern#1\dimexpr\macc@kerna}
\newcommand*\widebar[1]{\@ifnextchar^{{\wide@bar{#1}{0}}}{\wide@bar{#1}{1}}}
\newcommand*\wide@bar[2]{\if@single{#1}{\wide@bar@{#1}{#2}{1}}{\wide@bar@{#1}{#2}{2}}}
\newcommand*\wide@bar@[3]{%
  \begingroup
  \def\mathaccent##1##2{%
    \if#32 \let\macc@nucleus\first@char \fi
    \setbox\z@\hbox{$\macc@style{\macc@nucleus}_{}$}%
    \setbox\tw@\hbox{$\macc@style{\macc@nucleus}{}_{}$}%
    \dimen@\wd\tw@
    \advance\dimen@-\wd\z@
    \divide\dimen@ 3
    \@tempdima\wd\tw@
    \advance\@tempdima-\scriptspace
    \divide\@tempdima 10
    \advance\dimen@-\@tempdima
    \ifdim\dimen@>\z@ \dimen@0pt\fi
    \rel@kern{0.6}\kern-\dimen@
    \if#31
      \overline{\rel@kern{-0.6}\kern\dimen@\macc@nucleus\rel@kern{0.4}\kern\dimen@}%
      \advance\dimen@0.4\dimexpr\macc@kerna
      \let\final@kern#2%
      \ifdim\dimen@<\z@ \let\final@kern1\fi
      \if\final@kern1 \kern-\dimen@\fi
    \else
      \overline{\rel@kern{-0.6}\kern\dimen@#1}%
    \fi
  }%
  \macc@depth\@ne
  \let\math@bgroup\@empty \let\math@egroup\macc@set@skewchar
  \mathsurround\z@ \frozen@everymath{\mathgroup\macc@group\relax}%
  \macc@set@skewchar\relax
  \let\mathaccentV\macc@nested@a
  \if#31
    \macc@nested@a\relax111{#1}%
  \else
    \def\gobble@till@marker##1\endmarker{}%
    \futurelet\first@char\gobble@till@marker#1\endmarker
    \ifcat\noexpand\first@char A\else
      \def\first@char{}%
    \fi
    \macc@nested@a\relax111{\first@char}%
  \fi
  \endgroup
}
\makeatother

\mapcitekey{Saito:HodgeModules}{Saito-HM}
\mapcitekey{Saito:MixedHodgeModules}{Saito-MHM}
\mapcitekey{Saito:Theory}{Saito-th}
\mapcitekey{Saito:Kollar}{Saito-K}
\mapcitekey{Saito:Kaehler}{Saito-Kae}
\mapcitekey{Schnell:sanya}{Schnell-sanya}
\mapcitekey{Schnell:holonomic}{Schnell-holo}
\mapcitekey{Takegoshi:HigherDirectImages}{Takegoshi}
\mapcitekey{Wang:TorsionPoints}{Wang}
\mapcitekey{Schnell:lazarsfeld}{Schnell-laz}
\mapcitekey{Popa+Schnell:mhmgv}{PS}
\mapcitekey{Green+Lazarsfeld:GV1}{GL1}
\mapcitekey{Green+Lazarsfeld:GV2}{GL2}
\mapcitekey{Simpson:Subspaces}{Simpson}
\mapcitekey{Kollar:DualizingII}{Kollar}
\mapcitekey{Beilinson+Bernstein+Deligne}{BBD}
\mapcitekey{Deligne:Finitude}{Deligne}
\mapcitekey{Deligne:HodgeII}{Deligne-H}
\mapcitekey{Lazarsfeld+Popa+Schnell:BGG}{LPS}
\mapcitekey{Schmid:VHS}{Schmid}
\mapcitekey{Chen+Jiang:VanishingEulerCharacteristic}{ChenJiang}
\mapcitekey{Ueno:ClassificationTheory}{Ueno}
\mapcitekey{Schnell:saito-vanishing}{Schnell-van}
\mapcitekey{Pareschi+Popa:cdf}{PP1}
\mapcitekey{Pareschi+Popa:regIII}{PP2}
\mapcitekey{Pareschi+Popa:GV}{PP3}
\mapcitekey{Pareschi+Popa:regI}{PP4}
\mapcitekey{Popa:perverse}{Popa}
\mapcitekey{Ben-Bassat+Block+Pantev:noncomm_FM}{BBP}
\mapcitekey{Pareschi:BasicResults}{Pareschi}
\mapcitekey{Jiang:effective}{jiang}
\mapcitekey{Varouchas:image}{Varouchas}
\mapcitekey{Ein+Lazarsfeld:Theta}{EL}
\mapcitekey{Chen+Hacon:abelian}{CH1}
\mapcitekey{Chen+Hacon:FiberSpaces}{CH2}
\mapcitekey{Lazarsfeld+Popa:BGG}{LP}
\mapcitekey{Hacon:GV}{Hacon}
\mapcitekey{Birkenhake+Lange:ComplexTori}{BL}
\mapcitekey{Schmid+Vilonen:Unitary}{SV}
\mapcitekey{Cattani+Kaplan+Schmid:Degeneration}{CKS}
\mapcitekey{Zucker:DegeneratingCoefficients}{Zucker}
\mapcitekey{Zhang:Quaternions}{Zhang}
\mapcitekey{Debarre:Ample}{Debarre}
\mapcitekey{Grauert+Peternell+Remmert:SCV7}{SCV}
\mapcitekey{Mourougane+Takayama:Metric}{MT}
\mapcitekey{Arapura:GV}{Arapura}
\mapcitekey{Chen+Hacon:Varieties}{CH-pisa}

\newcommand{\HM}[2]{\operatorname{HM}_{\RR}(#1, #2)}
\newcommand{\HMC}[2]{\operatorname{HM}_{\CC}(#1, #2)}
\newcommand{\Dbcoh}{\mathrm{D}_{\mathit{coh}}^{\mathit{b}}}

\newcommand{\Dbc}{\mathrm{D}_{\mathit{c}}^{\mathit{b}}}
\newcommand{\omX}{\omega_X}
\renewcommand{\DD}{\mathbf{D}}

\newcommand{\OT}{\shO_T}

\newcommand{\CCrho}{\CC_{\rho}}
\newcommand{\CCrhob}{\CC_{\rhob}}

\DeclareMathOperator{\Char}{Char}
\newcommand{\varH}{\mathcal{H}}
\newcommand{\varHR}{\varH_{\RR}}
\newcommand{\varHC}{\varH_{\CC}}

\newcommand{\rhob}{\bar\rho}

\newcommand{\ratM}{M_{\RR}}
\newcommand{\ratMC}{M_{\CC}}

\newcommand{\ratJC}{J_{\CC}}

\begin{document}

%========================================================
\title{Hodge modules on complex tori and generic vanishing for compact K\"ahler manifolds}

\author{Giuseppe Pareschi}
\address{Universit\`a di Roma ``Tor Vergata'', Dipartimento di Matematica,
Via della Ricerca Scientifica, I-00133 Roma, Italy}
\email{pareschi@axp.mat.uniroma2.it}

\author{Mihnea Popa}
\address{Department of Mathematics, Northwestern University,
2033 Sheridan Road, Evanston, IL 60208, USA} 
\email{mpopa@math.northwestern.edu}

\author{Christian Schnell}
\address{Department of Mathematics, Stony Brook University, Stony Brook, NY 11794-3651}
\email{cschnell@math.sunysb.edu}

\begin{abstract}
We extend the results of generic vanishing theory to polarizable real Hodge
modules on compact complex tori, and from there to arbitrary compact K\"ahler
manifolds. As applications, we obtain a bimeromorphic characterization of
compact complex tori (among compact K\"ahler manifolds), semi-positivity results, 
and a description of the Leray filtration for maps to tori.
\end{abstract}
\date{\today}
\maketitle

\markboth{GIUSEPPE PARESCHI, MIHNEA POPA, AND CHRISTIAN SCHNELL} 
{GENERIC VANISHING ON COMPACT K\"AHLER MANIFOLDS}

%========================================================

\section{Introduction}

The term ``generic vanishing'' refers to a collection of theorems about the
cohomology of line bundles with trivial first Chern class. The first results of this
type were obtained by Green and Lazarsfeld in the late 1980s \cite{GL1,GL2}; they
were proved using classical Hodge theory and are therefore valid on arbitrary compact
K\"ahler manifolds.  About ten years ago, Hacon \cite{Hacon} found a more algebraic
approach, using vanishing theorems and the Fourier-Mukai transform, that has led to
many additional results in the projective case; see also \cite{PP3,ChenJiang,PS}. The
purpose of this paper is to show that the newer results are in fact also valid on
arbitrary compact K\"ahler manifolds. 

Besides \cite{Hacon}, our motivation also comes from a 2013 paper by Chen and Jiang
\cite{ChenJiang} in which they prove, roughly speaking, that the direct image of the
canonical bundle under a generically finite morphism to an abelian variety is
semi-ample. Before we can state more precise results, recall the following useful
definitions (see \parref{par:GV-sheaves} for more details).

\begin{definition*}
Given a coherent $\OT$-module $\shF$ on a compact complex torus $T$, define
\[
	S^i(T, \shF) = \menge{L \in \Pic^0(T)}{H^i(T, \shF \tensor L) \neq 0}.
\]
We say that $\shF$ is a \define{GV-sheaf} if $\codim S^i(T, \shF) \geq i$ for every
$i \geq 0$; we say that $\shF$ is \define{M-regular} if $\codim S^i(T, \shF) \geq
i+1$ for every $i \geq 1$.
\end{definition*}

Hacon \cite[\S4]{Hacon} showed that if $f \colon X \to A$ is a morphism from a smooth
projective variety to an abelian variety, then the higher direct image sheaves $R^j
\fl \omX$ are GV-sheaves on $A$; in the special case where $f$ is generically finite
over its image, Chen and Jiang \cite[Theorem~1.2]{ChenJiang} proved the much stronger
result that $\fl \omX$ is, up to tensoring by line bundles in $\Pic^0(A)$, the direct
sum of pullbacks of M-regular sheaves from quotients of $A$. Since GV-sheaves are
nef, whereas M-regular sheaves are ample, one should think of this as saying that
$\fl \omX$ is not only nef but actually semi-ample. One of our main results is the
following generalization of this fact.

\begin{intro-theorem}\label{thm:direct_image}
Let $f \colon X \to T$ be a holomorphic mapping from a compact K\"ahler manifold to a
compact complex torus. Then for $j \geq 0$, one has a decomposition
\[
	R^j \fl \omX \simeq \bigoplus_{k=1}^n \bigl( \qu_k \shF_k \tensor L_k \bigr),
\]
where each $\shF_k$ is an M-regular (hence ample) coherent sheaf with projective support 
on the compact complex torus $T_k$, each $q_k \colon T \to T_k$ is a surjective morphism with connected fibers,
and each $L_k \in \Pic^0(T)$ has finite order. In particular, $R^j \fl \omX$ is a
GV-sheaf on $T$.
\end{intro-theorem}

This leads to quite strong positivity properties for higher direct images of
canonical bundles under maps to tori. For instance, if $f$ is a surjective map which
is a submersion away from a divisor with simple normal crossings, then $R^j \fl \omX$
is a semi-positive vector bundle on $T$.  See \parref{par:semi-positivity} for more
on this circle of ideas.

One application of \theoremref{thm:direct_image} is the following effective criterion
for a compact K\"ahler manifold to be bimeromorphically equivalent to a torus; this
generalizes a well-known theorem by Chen and Hacon in the projective case \cite{CH1}.

\begin{intro-theorem}\label{torus_intro}
A compact K\"ahler manifold $X$ is bimeromorphic to a compact complex
torus if and only if $\dim H^1(X, \CC) = 2 \dim X$ and $P_1(X) = P_2(X) = 1$.  
\end{intro-theorem}

The proof is inspired by the approach to the Chen-Hacon theorem given in
\cite{Pareschi}; even in the projective case, however, the result in
\corollaryref{cor:gen-finite} greatly simplifies the existing proof. In
\theoremref{thm:jiang}, we deduce that the Albanese map of a compact K\"ahler
manifold with $P_1(X) = P_2( X) = 1$ is surjective with connected fibers; in the
projective case, this was first proved by Jiang \cite{jiang}, as an effective version
of Kawamata's theorem about projective varieties of Kodaira dimension zero.
It is likely that the present methods can also be applied to the classification
of  compact K\"ahler manifolds with $\dim H^1(X, \CC) = 2\dim X$ and small
plurigenera; for the projective case, see for instance \cite{CH-pisa} and the
references therein.

In a different direction, \theoremref{thm:direct_image} combined with results in
\cite{LPS} leads to a concrete description of the Leray filtration on the cohomology
of $\omX$, associated with a holomorphic mapping $f \colon X \to T$ as above. Recall
that, for each $k \geq 0$, the Leray filtration is a decreasing filtration
$L^{\bullet} H^k(X, \omX)$ with the property that
\[
	\gr_L^i H^k(X, \omX) = H^i \bigl( T, R^{k-i} \fl \omX \bigr).
\]
One can also define a natural decreasing filtration $F^{\bullet} H^k (X, \omega_X)$ induced by the cup-product 
action of $H^1(T, \shO_T)$, namely 
\[
	F^i H^k(X, \omX) = \Im \left( \bigwedge^i H^1(T, \shO_T) \tensor H^{k-i} (X, \omX) \to H^{k}(X, \omX)\right).
\]

\begin{intro-theorem}\label{leray_intro}
The filtrations $L^{\bullet} H^k(X, \omX)$ and $F^{\bullet} H^k(X, \omX)$ coincide.
\end{intro-theorem}

A dual description of the filtration on global holomorphic forms is given in
\corollaryref{cor:Leray_forms}. Despite the elementary nature of the statement, we do
not know how to prove \theoremref{leray_intro} using only methods from classical
Hodge theory; finding a more elementary proof is an interesting problem.

Our approach to \theoremref{thm:direct_image} is to address generic vanishing for a
larger class of objects of Hodge-theoretic origin, namely polarizable real Hodge
modules on compact complex tori. This is not just a matter of higher generality;  we
do not know how to prove \theoremref{thm:direct_image} using methods of classical
Hodge theory in the spirit of \cite{GL1}. This is precisely due to the lack of an
\emph{a priori} description of the Leray filtration on $H^k (X, \omega_X)$ as in
\theoremref{leray_intro}.

The starting point for our proof of \theoremref{thm:direct_image} is a result by
Saito \cite{Saito-Kae}, which says that the coherent $\OT$-module $R^j \fl \omX$ is
part of a polarizable real Hodge module $M = (\Mmod, F_{\bullet} \Mmod, \ratM) \in
\HM{T}{\dim X + j}$ on the torus $T$; more precisely, 
\begin{equation} \label{eq:Saito}
	R^j \fl \omX \simeq \omega_T \tensor F_{p(M)} \Mmod
\end{equation}
is the first nontrivial piece in the Hodge filtration $F_{\bullet} \Mmod$ of the
underlying regular holonomic $\Dmod$-module $\Mmod$. (Please see \parref{par:RHM} for
some background on Hodge modules.) Note that $M$ is supported on the image
$f(X)$, and that its restriction to the smooth locus of $f$ is the polarizable
variation of Hodge structure on the $(\dim f + j)$-th cohomology of the fibers. The
reason for working with real coefficients is that the polarization is induced by a
choice of K\"ahler form in $H^2(X, \RR) \cap H^{1,1}(X)$; the variation of Hodge
structure itself is of course defined over $\ZZ$.

In light of \eqref{eq:Saito}, \theoremref{thm:direct_image} is a consequence of the
following general statement about polarizable real Hodge modules on compact complex tori.

\begin{intro-theorem} \label{thm:Chen-Jiang}
Let $M = (\Mmod, F_{\bullet} \Mmod, \ratM) \in \HM{T}{w}$ be a polarizable real Hodge
module on a compact complex torus $T$. Then for each $k \in \ZZ$, the coherent
$\OT$-module $\gr_k^F \Mmod$ decomposes as
\[
	\gr_k^F \Mmod \simeq \bigoplus_{j=1}^n 
		\bigl( \qu_j \shF_j \tensor_{\OT} L_j \bigr),
\]
where $q_j \colon T \to T_j$ is a surjective map with connected fibers to a 
complex torus, $\shF_j$ is an M-regular coherent sheaf on $T_j$ with projective
support, and $L_j \in \Pic^0(T)$. If $M$ admits an integral structure, then each
$L_j$ has finite order.
\end{intro-theorem}

Let us briefly describe the most important elements in the proof. In \cite{PS}, we
already exploited the relationship between generic vanishing and Hodge modules on
abelian varieties, but the proofs relied on vanishing theorems. What allows us to go
further is a beautiful new idea by Botong Wang \cite{Wang}, also dating to 2013,
namely that up to taking
direct summands and tensoring by unitary local systems, every polarizable real Hodge
module on a complex torus actually comes from an abelian variety. (Wang showed this
for Hodge modules of geometric origin.) This is a version with coefficients of Ueno's
result \cite{Ueno} that every irreducible subvariety of $T$ is a torus bundle over a
projective variety, and is proved by combining this geometric fact with some
arguments about variations of Hodge structure.

The existence of the decomposition in \theoremref{thm:Chen-Jiang} is due to the fact
that the regular holonomic $\Dmod$-module $\Mmod$ is semi-simple, hence isomorphic to
a direct sum of simple regular holonomic $\Dmod$-modules. This follows from a theorem
by Deligne and Nori \cite{Deligne}, which says that the local system underlying a
polarizable real variation of Hodge structure on a Zariski-open subset of a compact
K\"ahler manifold is semi-simple. It turns out that the decomposition of $\Mmod$ into
simple summands is compatible with the Hodge filtration $F_{\bullet} \Mmod$; in order
to prove this, we introduce the category of ``polarizable complex Hodge
modules'' (which are polarizable real Hodge modules together with an endomorphism 
whose square is minus the identity), and show that every simple summand of $\Mmod$
underlies a polarizable complex Hodge module in this sense.

\begin{note}
Our ad-hoc definition of complex Hodge modules is good enough for the purposes of
this paper. As of 2016, a more satisfactory treatment, in terms of
$\Dmod$-modules and distribution-valued pairings, is currently being developed by
Claude Sabbah and the third author. The reader is advised to consult the website
\begin{center}
\url{www.cmls.polytechnique.fr/perso/sabbah.claude/MHMProject/mhm.html}
\end{center}
for more information.
\end{note}

The M-regularity of the individual summands in \theoremref{thm:Chen-Jiang} turns out
to be closely related to the Euler characteristic of the corresponding
$\Dmod$-modules. The results in \cite{PS} show that when $(\Mmod, F_{\bullet} \Mmod)$
underlies a polarizable complex Hodge module on an abelian variety $A$, the Euler
characteristic satisfies $\chi(A, \Mmod) \geq 0$, and each coherent $\OA$-module
$\gr_k^F \Mmod$ is a GV-sheaf. The new result (in \lemmaref{lem:M-regular}) is that
each $\gr_k^F \Mmod$ is actually M-regular, provided that $\chi(A, \Mmod) > 0$. 
That we can always get into the situation where the Euler characteristic is positive
follows from some general results about simple holonomic $\Dmod$-modules from
\cite{Schnell-holo}. 

\theoremref{thm:Chen-Jiang} implies that each graded quotient $\gr_k^F \Mmod$ with
respect to the Hodge filtration is a GV-sheaf, the K\"ahler analogue of a result in
\cite{PS}. However, the stronger formulation above is new even in the case of smooth
projective varieties, and has further useful consequences. One such is the following:
for a holomorphic mapping $f \colon X \to T$ that is generically finite onto its
image, the locus
\[
	S^0 (T, f_* \omega_X) = \menge{L \in \Pic^0(T)}%
		{H^i(T, f_*\omega_X \tensor_{\OT} L) \neq 0}
\] 
is preserved by the involution $L \mapsto L^{-1}$ on $\Pic^0(T)$; see \corollaryref{cor:gen-finite}. 
This is a crucial ingredient in the proof of \theoremref{torus_intro}.

Going back to Wang's paper \cite{Wang}, its main purpose was to prove Beauville's
conjecture, namely that on a compact K\"ahler manifold $X$, every irreducible
component of every $\Sigma^k(X) = \menge{\rho \in \Char(X)}{H^k(X, \CCrho) \neq
0}$ contains characters of finite order. In the projective case, this is of course a
famous theorem by Simpson \cite{Simpson}. Combining the structural
\theoremref{thm:CHM-main} with known results about Hodge modules on abelian varieties
\cite{Schnell-laz} allows us to prove the following generalization of Wang's theorem
(which dealt with Hodge modules of geometric origin).

\begin{intro-theorem} \label{thm:finite-order}
If a polarizable real Hodge module $M  \in \HM{T}{w}$ on a compact complex torus
admits an integral structure, then the sets
\[
	S_m^i(T, M) 
		= \menge{\rho \in \Char(T)}{\dim H^i(T, \ratM \tensor_{\RR} \CCrho) \geq m}
\]
are finite unions of translates of linear subvarieties by points of finite order.
\end{intro-theorem}

The idea is to use Kronecker's theorem (about algebraic integers all of whose
conjugates have absolute value one) to prove that certain characters have finite
order. Roughly speaking, the characters in question are unitary because of the
existence of a polarization on $M$, and they take values in the group of algebraic
integers because of the existence of an integral structure on $M$.

\noindent
{\bf Projectivity questions.}
 We conclude by noting that many of the results in this paper can be placed in the broader context of the following problem:
 how far are natural geometric or sheaf theoretic constructions on compact K\"ahler manifolds in general, 
 and on compact complex tori in particular, from being determined by similar constructions on projective 
 manifolds? Theorems \ref{thm:direct_image} and \ref{thm:Chen-Jiang} provide the answer on tori
 in the case of Hodge-theoretic constructions. We thank J. Koll\'ar for suggesting this point of view, and also the 
statements of the problems in the paragraph below.
 
Further structural results could provide a general machine for reducing certain questions about K\"ahler manifolds to 
the algebraic setting. For instance, by analogy with positivity conjectures in the algebraic case, one hopes for the following result
in the case of varying families: if $X$ and $Y$ are compact K\"ahler manifolds and 
 $f: X \rightarrow Y$ is a fiber space of maximal variation,  i.e. such that the general fiber is bimeromorphic 
 to at most countably many other fibers, then $Y$ is projective. More generally, for an arbitrary such $f$, 
 is there a mapping $g: Y \rightarrow Z$ with $Z$ projective, such that the fibers of $f$ are bimeromorphically 
 isotrivial over those of $Y$?

A slightly more refined version in the case when $Y = T$ is a torus, which is essentially a combination of Iitaka fibrations and Ueno's conjecture, 
is this: there should exist a morphism $h: X \rightarrow Z$, where $Z$ is a variety of general type 
generating an abelian quotient $g : T \rightarrow A$, such that the fibers of $h$ have Kodaira dimension $0$ and 
are bimeromorphically isotrivial over the fibers of $g$.

\section{Real and complex Hodge modules}

\subsection{Real Hodge modules}
\label{par:RHM}

In this paper, we work with polarizable real Hodge modules on complex manifolds. This
is the natural setting for studying compact K\"ahler manifolds, because the
polarizations induced by K\"ahler forms are defined over $\RR$ (but usually not
over $\QQ$, as in the projective case).  Saito originally developed the theory of
Hodge modules with rational coefficients, but as explained in \cite{Saito-Kae},
everything works just as well with real coefficients, provided one relaxes the
assumptions about local monodromy: the eigenvalues of the monodromy operator on the
nearby cycles are allowed to be arbitrary complex numbers of absolute value one,
rather than just roots of unity.  This has already been observed several times in the
literature \cite{SV}; the point is that Saito's theory rests on certain results
about polarizable variations of Hodge structure \cite{Schmid,Zucker,CKS}, which hold
in this generality.

Let $X$ be a complex manifold. We first recall some terminology. 

\begin{definition}
We denote by $\HM{X}{w}$ the category of polarizable real Hodge modules of weight
$w$; this is a semi-simple $\RR$-linear abelian category, endowed with a faithful functor
to the category of real perverse sheaves. 
\end{definition}

Saito constructs $\HM{X}{w}$ as a full subcategory of the category of all filtered
regular holonomic $\Dmod$-modules with real structure, in several stages. To begin
with, recall that a \define{filtered regular holonomic $\Dmod$-module with real
structure} on $X$ consists of the following four pieces of data: (1) a regular
holonomic left $\Dmod_X$-module $\Mmod$; (2) a good filtration $F_{\bullet} \Mmod$ by
coherent $\OX$-modules; (3) a perverse sheaf $\ratM$ with coefficients in $\RR$; (4)
an isomorphism $\ratM \tensor_{\RR} \CC \simeq \DR(\Mmod)$. Although the isomorphism
is part of the data, we usually suppress it from the notation and simply
write $M = (\Mmod, F_{\bullet} \Mmod, \ratM)$.  The \define{support} $\Supp M$ is
defined to be the support of the underlying perverse sheaf $\ratM$; one says that $M$
has \define{strict support} if $\Supp M$ is irreducible and if $M$ has no nontrivial
subobjects or quotient objects that are supported on a proper subset of $\Supp M$. 

Now $M$ is called a \define{real Hodge module of weight $w$} if it satisfies several
additional conditions that are imposed by recursion on the dimension of $\Supp M$.
Although they are not quite stated in this way in \cite{Saito-HM}, the essence of
these conditions is that (1) every Hodge module decomposes into a sum of Hodge
modules with strict support,
and (2) every Hodge module with strict support is generically a variation of Hodge
structure, which uniquely determines the Hodge module. Given $k \in \ZZ$, set $\RR(k)
= (2\pi i)^k \RR \subseteq \CC$; then one has the \define{Tate twist}
\[
	M(k) = \bigl( \Mmod, F_{\bullet - k} \Mmod, \ratM \tensor_{\RR} \RR(k) \bigr)
		\in \HM{X}{w-2k}.
\]
Every real Hodge module of weight $w$ has a well-defined \define{dual} $\DD M$,
which is a real Hodge module of weight $-w$ whose underlying perverse sheaf is the
Verdier dual $\DD \ratM$. A \define{polarization} is an isomorphism of real Hodge
modules $\DD M \simeq M(w)$, subject to certain conditions that are again imposed
recursively; one says that $M$ is \define{polarizable} if it admits at least one
polarization. 

\begin{example}
Every polarizable real variation of Hodge structure of weight $w$ on $X$ gives rise
to an object of $\HM{X}{w + \dim X}$. If $\varH$ is such a variation, we denote the
underlying real local system by $\varHR$, its complexification by $\varHC = \varHR
\tensor_{\RR} \CC$, and the corresponding flat bundle by $(\varH, \nabla)$; then
$\varH \simeq \varHC \tensor_{\CC} \OX$. The flat connection makes $\varH$ into a
regular holonomic left $\Dmod$-module, filtered by $F_{\bullet} \varH = F^{-\bullet}
\varH$; the real structure is given by the real perverse sheaf $\varHR \decal{\dim X}$.
\end{example}

We list a few useful properties of polarizable real Hodge modules. By definition,
every object $M \in \HM{X}{w}$ admits a locally finite \define{decomposition by
strict support}; when $X$ is compact, this is a finite decomposition
\[
	M \simeq \bigoplus_{j=1}^n M_j,
\]
where each $M_j \in \HM{X}{w}$ has strict support equal to an irreducible analytic
subvariety $Z_j \subseteq X$. There are no nontrivial morphisms between Hodge modules
with different strict support; if we assume that $Z_1, \dotsc, Z_n$ are distinct, the
decomposition by strict support is therefore unique. Since the category $\HM{X}{w}$
is semi-simple, it follows that every polarizable real Hodge module of weight $w$ is
isomorphic to a direct sum of simple objects with strict support.

One of Saito's most important results is the following structure theorem relating
polarizable real Hodge modules and polarizable real variations of Hodge structure.

\begin{theorem}[Saito] \label{thm:Saito-structure}
The category of polarizable real Hodge modules of weight $w$ with strict support $Z
\subseteq X$ is equivalent to the category of generically defined polarizable real
variations of Hodge structure of weight $w - \dim Z$ on $Z$.
\end{theorem}

In other words, for any $M \in \HM{X}{w}$ with strict support $Z$, there is a dense
Zariski-open subset of the smooth locus of $Z$ over which it restricts to a
polarizable real variation of Hodge structure; conversely, every such variation
extends uniquely to a Hodge module with strict support $Z$. The proof in
\cite[Theorem~3.21]{Saito-MHM} carries over to the case of real coefficients; see
\cite{Saito-Kae} for further discussion.

\begin{lemma} \label{lem:Kashiwara}
The support of $M \in \HM{X}{w}$ lies in a submanifold $i \colon Y \into
X$ if and only if $M$ belongs to the image of the functor $\il \colon \HM{Y}{w} \to
\HM{X}{w}$.
\end{lemma}

This result is often called \define{Kashiwara's equivalence}, because Kashiwara
proved the same thing for arbitrary coherent $\Dmod$-modules. In the case of
Hodge modules, the point is that the coherent $\OX$-modules $F_k \Mmod /
F_{k-1} \Mmod$ are in fact $\OY$-modules.

\subsection{Compact K\"ahler manifolds and semi-simplicity}

In this section, we prove some results about the underlying regular holonomic
$\Dmod$-modules of polarizable real Hodge modules on compact K\"ahler manifolds.
Our starting point is the following semi-simplicity theorem, due to
Deligne and Nori.

\begin{theorem}[Deligne, Nori]
Let $X$ be a compact K\"ahler manifold. If 
\[
	M = (\Mmod, F_{\bullet} \Mmod, \ratM) \in \HM{X}{w}, 
\]
then the perverse sheaf $\ratM$ and the $\Dmod$-module $\Mmod$ are
semi-simple.
\end{theorem}

\begin{proof}
Since the category $\HM{X}{w}$ is semi-simple, we may assume without loss of
generality that $M$ is simple, with strict support an irreducible analytic subvariety
$Z \subseteq X$. By Saito's \theoremref{thm:Saito-structure}, $M$ restricts to a
polarizable real variation of Hodge structure $\varH$ of weight $w - \dim Z$ on a
Zariski-open subset of the smooth locus of $Z$; note that $\varH$ is a simple object
in the category of real variations of Hodge structure. Now $\ratM$
is the intersection complex of $\varHR$, and so it suffices to prove that $\varHR$ is
semi-simple. After resolving singularities, we can assume that $\varH$ is defined on
a Zariski-open subset of a compact K\"ahler manifold; in that case, Deligne and Nori
have shown that $\varHR$ is semi-simple \cite[\S1.12]{Deligne}. It follows that the
complexification $\ratM \tensor_{\RR} \CC$ of the perverse sheaf is semi-simple as
well; by the Riemann-Hilbert correspondence, the same is true for the underlying
regular holonomic $\Dmod$-module $\Mmod$.  
\end{proof}

A priori, there is no reason why the decomposition of the regular holonomic
$\Dmod$-module $\Mmod$ into simple factors should lift to a decomposition in the
category $\HM{X}{w}$. Nevertheless, it turns out that we can always chose the
decomposition in such a way that it is compatible with the filtration $F_{\bullet}
\Mmod$.

\begin{proposition} \label{prop:complexification}
Let $M \in \HM{X}{w}$ be a simple polarizable real Hodge module on a compact K\"ahler
manifold. Then one of the following two statements is true:
\begin{enumerate}
\item The underlying perverse sheaf $\ratM \tensor_{\RR} \CC$ is simple.
\item There is an endomorphism $J \in \End(M)$ with $J^2 = -\id$ such that
\[	
	\bigl( \Mmod, F_{\bullet} \Mmod, \ratM \tensor_{\RR} \CC \bigr)
	= \ker(J - i \cdot \id) \oplus \ker(J + i \cdot \id),
\]
and the perverse sheaves underlying $\ker(J \pm i \cdot \id)$ are simple.
\end{enumerate}
\end{proposition}

We begin by proving the following key lemma.

\begin{lemma} \label{lem:simple-RVHS}
Let $\varH$ be a polarizable real variation of Hodge structure on a Zariski-open
subset of a compact K\"ahler manifold. If $\varH$ is simple, then
\begin{aenumerate}
\item either the underlying complex local system $\varHC$ is also simple,
\item or there is an endomorphism $J \in \End(\varH)$ with $J^2 = -\id$, such that
\[
	\varHC = \ker(\ratJC - i \cdot \id) \oplus \ker(\ratJC + i \cdot \id)
\]
is the sum of two (possibly isomorphic) simple local systems.
\end{aenumerate}
\end{lemma}

\begin{proof}
Since $X$ is a Zariski-open subset of a compact K\"ahler manifold, the theorem of the
fixed part holds on $X$, and the local system $\varHC$ is semi-simple
\cite[\S1.12]{Deligne}. Choose a base point $x_0 \in X$, and write $\HR$ for the
fiber of the local system $\varHR$ at the point $x_0$; it carries a polarizable 
Hodge structure
\[
	\HC = \HR \tensor_{\RR} \CC = \bigoplus_{p+q=w} H^{p,q},
\]
say of weight $w$. The fundamental group $\Gamma = \pi_1(X, x_0)$ acts on $\HR$, and as
we remarked above, $\HC$ decomposes into a sum of simple $\Gamma$-modules. The proof
of \cite[Proposition~1.13]{Deligne} shows that there is a nontrivial simple
$\Gamma$-module $V \subseteq \HC$ compatible with the Hodge decomposition, meaning
that
\[
	V = \bigoplus_{p+q=w} V \cap H^{p,q}.
\]
Let $\bar{V} \subseteq \HC$ denote the conjugate of $V$ with respect to the real
structure $\HR$; it is another nontrivial simple $\Gamma$-module with 
\[
	\bar{V} = \bigoplus_{p+q=w} \bar{V} \cap H^{p,q}.
\]
The intersection $(V + \bar{V}) \cap \HR$ is therefore a $\Gamma$-invariant real
sub-Hodge structure of $\HR$. By the theorem of the fixed part, it extends to a real
sub-variation of $\varH$; since $\varH$ is simple, this means
that $\HC = V + \bar{V}$. Now there are two possibilities. (1) If $V = \bar{V}$, then
$\HC = V$, and $\varHC$ is a simple local system. (2) If $V \neq \bar{V}$, then $\HC
= V \oplus \bar{V}$, and $\varHC$ is the sum of two (possibly isomorphic) simple local
systems. The endomorphism algebra $\End(\varHR)$ coincides with the subalgebra of
$\Gamma$-invariants in $\End(\HR)$; by the theorem of the fixed part, it is also a real
sub-Hodge structure. Let $p \in \End(\HC)$ and $\bar{p} \in \End(\HC)$ denote the
projections to the two subspaces $V$ and $\bar{V}$; both preserve the Hodge
decomposition, and are therefore of type $(0,0)$. This shows that the element $J =
i(p - \bar{p}) \in \End(\HC)$ is a real Hodge class of type $(0,0)$ with $J^2 =
-\id$; by the theorem of the fixed part, $J$ is the restriction to $x_0$ of an
endomorphism of the variation of Hodge structure $\varH$. This completes the proof
because $V$ and $\bar{V}$ are exactly the $\pm i$-eigenspaces of $J$.
\end{proof}

\begin{proof}[Proof of \propositionref{prop:complexification}]
Since $M$ is simple, it has strict support equal to an irreducible analytic
subvariety $Z \subseteq X$; by \theoremref{thm:Saito-structure}, $M$ is obtained from
a polarizable real variation of Hodge structure $\varH$ of weight $w - \dim Z$ on a
dense Zariski-open subset of the smooth locus of $Z$. Let $\varHR$ denote the
underlying real local system; then $\ratM$ is isomorphic to the intersection complex
of $\varHR$.  Since we can resolve the singularities of $Z$ by blowing up along
submanifolds of $X$, \lemmaref{lem:simple-RVHS} applies to this situation; it shows
that $\varHC = \varHR \tensor_{\RR} \CC$ has at most two simple factors. The same is
true for $\ratM \tensor_{\RR} \CC$ and, by the Riemann-Hilbert correspondence, for
$\Mmod$.

Now we have to consider two cases. If $\varHC$ is simple, then $\Mmod$ is also
simple, and we are done. If $\varHC$ is not simple, then by
\lemmaref{lem:simple-RVHS}, there is an endomorphism $J \in \End(\varH)$ with $J^2 =
-\id$ such that the two simple factors are the $\pm i$-eigenspaces of $J$. By
\theoremref{thm:Saito-structure}, it extends uniquely to an endomorphism of $J \in
\End(M)$ in the category $\HM{X}{w}$; in particular, we obtain an induced endomorphism 
\[
	J \colon \Mmod \to \Mmod
\]
that is strictly compatible with the filtration $F_{\bullet} \Mmod$ by
\cite[Proposition~5.1.14]{Saito-HM}. Now the $\pm i$-eigenspaces of $J$
give us the desired decomposition
\[
	(\Mmod, F_{\bullet} \Mmod) =
		(\Mmod', F_{\bullet} \Mmod') \oplus (\Mmod'', F_{\bullet} \Mmod'');
\]
note that the two regular holonomic $\Dmod$-modules $\Mmod'$ and $\Mmod''$ are simple
because the corresponding perverse sheaves are the intersection complexes of the
simple complex local systems $\ker(\ratJC \pm i \cdot \id)$, where $\ratJC$ stands
for the induced endomorphism of the complexification $\ratM \tensor_{\RR} \CC$.
\end{proof}

\subsection{Complex Hodge modules}
\label{par:CHM}

In Saito's recursive definition of the category of polarizable Hodge modules, the
existence of a real structure is crucial: to say that a given filtration on a complex
vector space is a Hodge structure of a certain weight, or that a given bilinear form
is a polarization, one needs to have complex conjugation. This explains why there is
as yet no general theory of ``polarizable complex Hodge modules'' -- although it
seems likely that such a theory can be constructed within the framework of twistor
$\Dmod$-modules developed by Sabbah and Mochizuki. We now explain a workaround for
this problem, suggested by \propositionref{prop:complexification}.

\begin{definition}
A \define{polarizable complex Hodge module} on a complex manifold $X$ is a pair
$(M, J)$, consisting of a polarizable real Hodge module $M \in \HM{X}{w}$ and an
endomorphism $J \in \End(M)$ with $J^2 = -\id$. 
\end{definition}

The space of morphisms between two polarizable complex Hodge modules $(M_1, J_1)$ and
$(M_2, J_2)$ is defined in the obvious way: 
\[
	\Hom \bigl( (M_1, J_1), (M_2, J_2) \bigr) = 
		\menge{f \in \Hom(M_1, M_2)}{f \circ J_1 = J_2 \circ f}
\]
Note that composition with $J_1$ (or equivalently, $J_2$) puts a natural complex
structure on this real vector space.

\begin{definition}
We denote by $\HMC{X}{w}$ the category of polarizable complex Hodge modules of weight
$w$; it is $\CC$-linear and abelian.
\end{definition}

From a polarizable complex Hodge module $(M, J)$, we obtain a filtered regular
holonomic $\Dmod$-module as well as a complex perverse sheaf, as follows. Denote by
\[
	\Mmod = \Mmod' \oplus \Mmod'' = \ker(J - i \cdot \id) \oplus \ker(J + i \cdot \id)
\]
the induced decomposition of the regular holonomic $\Dmod$-module underlying $M$, and
observe that $J \in \End(\Mmod)$ is strictly compatible with the Hodge filtration
$F_{\bullet} \Mmod$. This means that we have a decomposition
\[
	(\Mmod, F_{\bullet} \Mmod) = (\Mmod', F_{\bullet} \Mmod')
		\oplus (\Mmod'', F_{\bullet} \Mmod'')
\]
in the category of filtered $\Dmod$-modules. Similarly, let $\ratJC \in \End(\ratMC)$
denote the induced endomorphism of the complex perverse sheaf underlying $M$; then 
\[
	\ratMC = \ratM \tensor_{\RR} \CC =
		\ker(\ratJC - i \cdot \id) \oplus \ker(\ratJC + i \cdot \id),
\]
and the two summands correspond to $\Mmod'$ and $\Mmod''$ under the Riemann-Hilbert
correspondence. Note that they are isomorphic as \emph{real} perverse
sheaves; the only difference is in the $\CC$-action. We obtain a functor
\[
	(M, J) \mapsto \ker(\ratJC - i \cdot \id)
\]
from $\HMC{X}{w}$ to the category of complex perverse sheaves on $X$; it is faithful,
but depends on the choice of $i$. 

\begin{definition}
Given $(M, J) \in \HMC{X}{w}$, we call
\[
	\ker(\ratJC - i \cdot \id) \subseteq \ratMC
\]
the \define{underlying complex perverse sheaf}, and
\[
	(\Mmod', F_{\bullet} \Mmod') = \ker(J - i \cdot \id) 
		\subseteq (\Mmod, F_{\bullet} \Mmod)
\]
the \define{underlying filtered regular holonomic $\Dmod$-module}.
\end{definition}

There is also an obvious functor from polarizable real Hodge modules to polarizable
complex Hodge modules: it takes $M \in \HM{X}{w}$ to the pair 
\[
	\bigl( M \oplus M, J_M \bigr), 
		\quad J_M(m_1, m_2) = (-m_2, m_1).
\]
Not surprisingly, the underlying complex perverse sheaf is isomorphic to $\ratM
\tensor_{\RR} \CC$, and the underlying filtered regular holonomic $\Dmod$-module to
$(\Mmod, F_{\bullet} \Mmod)$. The proof of the following lemma is left as an easy exercise.

\begin{lemma}
A polarized complex Hodge module $(M, J) \in \HMC{X}{w}$ belongs to the image of
$\HM{X}{w}$ if and only if there exists $r \in \End(M)$ with 
\[
	r \circ J = -J \circ r \quad \text{and} \quad r^2 = \id.
\]
\end{lemma}

In particular, $(M, J)$ should be isomorphic to its \define{complex conjugate} $(M,
-J)$, but this in itself does not guarantee the existence of a real structure -- for
example when $M$ is simple and $\End(M)$ is isomorphic to the quaternions $\HH$. 

\begin{proposition} \label{prop:semi-simple}
The category $\HMC{X}{w}$ is semi-simple, and the simple objects are of the following
two types:
\begin{renumerate}
\item \label{en:simple-1}
$(M \oplus M, J_M)$, where $M \in \HM{X}{w}$ is simple and $\End(M) = \RR$.
\item \label{en:simple-2}
$(M, J)$, where $M \in \HM{X}{w}$ is simple and $\End(M) \in \{\CC, \HH\}$.
\end{renumerate}
\end{proposition}

\begin{proof}
Since $\HM{X}{w}$ is semi-simple, every object of $\HMC{X}{w}$ is isomorphic to a
direct sum of polarizable complex Hodge modules of the form
\begin{equation} \label{eq:semi-simple}
	\bigl( M^{\oplus n}, J \bigr),
\end{equation}
where $M \in \HM{X}{w}$ is simple, and $J$ is an $n \times n$-matrix with entries in
$\End(M)$ such that $J^2 = -\id$. By Schur's lemma and the classification of real
division algebras, the endomorphism algebra of a
simple polarizable real Hodge module is one of $\RR$, $\CC$, or $\HH$. If
$\End(M) = \RR$, elementary linear algebra shows that $n$ must be even and that
\eqref{eq:semi-simple} is isomorphic to the direct sum of $n/2$ copies of
\ref{en:simple-1}. If $\End(M) = \CC$, one can diagonalize the matrix $J$; this
means that \eqref{eq:semi-simple} is isomorphic to a direct sum of $n$ objects of
type \ref{en:simple-2}. If $\End(M) = \HH$, it is still possible to diagonalize
$J$, but this needs some nontrivial results about matrices with entries in the quaternions
\cite{Zhang}. Write $J \in M_n(\HH)$ in the form $J = J_1 + J_2 j$, with $J_1, J_2
\in M_n(\CC)$, and consider the ``adjoint matrix''
\[
	\chi_J = \begin{pmatrix}
		J_1 & J_2 \\ - \overline{J_2} & \overline{J_1} 
	\end{pmatrix} \in M_{2n}(\CC).
\]
Since $J^2 = -\id$, one also has $\chi_J^2 = -\id$, and so the matrix $J$ is normal
by \cite[Theorem~4.2]{Zhang}. According to \cite[Corollary~6.2]{Zhang}, this implies
the existence of a unitary matrix $U \in M_n(\HH)$ such that $U^{-1} A U = i \cdot
\id$; here unitary means that $U^{-1} = U^{\ast}$ is equal to the conjugate transpose 
of $U$. The consequence is that \eqref{eq:semi-simple} is again isomorphic to a
direct sum of $n$ objects of type \ref{en:simple-2}. Since it is straightforward to
prove that both types of objects are indeed simple, this concludes the proof.
\end{proof}

\begin{note}
The three possible values for the endomorphism algebra of a simple object $M \in
\HM{X}{w}$ reflect the splitting behavior of its complexification $(M \oplus M, J_M)
\in \HMC{X}{w}$: if $\End(M) = \RR$, it remains irreducible; if $\End(M) = \CC$, it
splits into two non-isomorphic simple factors; if $\End(M) = \HH$, it splits into two
isomorphic simple factors. Note that the endomorphism ring of a simple polarizable
complex Hodge module is always isomorphic to $\CC$, in accordance with Schur's lemma. 
\end{note}

Our ad-hoc definition of the category $\HMC{X}{w}$ has the advantage that every
result about polarizable real Hodge modules that does not explicitly mention the
real structure extends to polarizable complex Hodge modules. For example, each $(M,
J) \in \HMC{X}{w}$ admits a unique decomposition by strict support: $M$ admits such a
decomposition, and since there are no nontrivial morphisms between objects with
different strict support, $J$ is automatically compatible with the decomposition.
For much the same reason, Kashiwara's equivalence (in \lemmaref{lem:Kashiwara}) holds
also for polarizable complex Hodge modules. 

Another result that immediately carries over is Saito's direct image theorem. The
strictness of the direct image complex is one of the crucial properties of polarizable
Hodge modules; in the special case of the morphism from a projective variety $X$ to a
point, it is equivalent to the $E_1$-degeneration of the spectral sequence
\[
	E_1^{p,q} = H^{p+q} \bigl( X, \gr_p^F \DR(\Mmod') \bigr)
		\Longrightarrow H^{p+q} \bigl( X, \DR(\Mmod') \bigr),
\]
a familiar result in classical Hodge theory when $\Mmod' = \OX$.

\begin{theorem}
Let $f \colon X \to Y$ be a projective morphism between complex manifolds. 
\begin{aenumerate}
\item If $(M, J) \in \HMC{X}{w}$, then for each $k \in \ZZ$, the pair
\[
	\shH^k \fl(M, J) = \bigl( \shH^k \fl M, \shH^k \fl J \bigr) \in \HMC{Y}{w+k}
\]
is again a polarizable complex Hodge module. 
\item The direct image complex $\fp(\Mmod', F_{\bullet} \Mmod')$ is strict, and
$\shH^k \fp(\Mmod', F_{\bullet} \Mmod')$ is the filtered regular holonomic
$\Dmod$-module underlying $\shH^k \fl(M, J)$.
\end{aenumerate}
\end{theorem}

\begin{proof}
Since $M \in \HM{X}{w}$ is a polarizable real Hodge module, we have $\shH^k \fl M \in
\HM{Y}{w+k}$ by Saito's direct image theorem \cite[Th\'eor\`eme~5.3.1]{Saito-HM}. Now
it suffices to note that $J \in \End(M)$ induces an endomorphism $\shH^k \fl J \in
\End \bigl( \shH^k \fl M \bigr)$ whose square is equal to minus the identity. Since
\[
	(\Mmod, F_{\bullet} \Mmod) = (\Mmod', F_{\bullet} \Mmod') \oplus
		(\Mmod'', F_{\bullet} \Mmod''),
\]
the strictness of the complex $\fp(\Mmod', F_{\bullet} \Mmod')$ follows from that of
$\fp(\Mmod, F_{\bullet} \Mmod)$, which is part of the above-cited theorem by Saito.
\end{proof}

On compact K\"ahler manifolds, the semi-simplicity results from the previous
section can be summarized as follows.

\begin{proposition} \label{prop:CHM-Kaehler}
Let $X$ be a compact K\"ahler manifold. 
\begin{aenumerate}
\item A polarizable complex Hodge module $(M, J) \in \HMC{X}{w}$ is simple if and
only if the underlying complex perverse sheaf 
\[
	\ker \Bigl( \ratJC - i \cdot \id \, \colon 
		\ratM \tensor_{\RR} \CC \to \ratM \tensor_{\RR} \CC \Bigr)
\]
is simple. 
\item If $M \in \HM{X}{w}$, then every simple factor of the complex perverse sheaf
$\ratM \tensor_{\RR} \CC$ underlies a polarizable complex Hodge module.
\end{aenumerate}
\end{proposition}

\begin{proof}
This is a restatement of \propositionref{prop:complexification}.
\end{proof}

\subsection{Complex variations of Hodge structure}

In this section, we discuss the relation between polarizable complex Hodge modules
and polarizable complex variations of Hodge structure. 

\begin{definition}
A \define{polarizable complex variation of Hodge structure} is a pair $(\varH, J)$,
where $\varH$ is a polarizable real variation of Hodge structure, and $J \in
\End(\varH)$ is an endomorphism with $J^2 = -\id$.
\end{definition}

As before, the \define{complexification} of a real variation $\varH$ is defined as
\[
	\bigl( \varH \oplus \varH, J_{\varH} \bigr), \quad
		J_{\varH}(h_1, h_2) = (-h_2, h_1),
\]
and a complex variation $(\varH, J)$ is real if and only if there is an endomorphism
$r \in \End(\varH)$ with $r \circ J = - J \circ r$ and $r^2 = \id$. Note that the
direct sum of $(\varH, J)$ with its \define{complex conjugate} $(\varH, -J)$ has an
obvious real structure.

The definition above is convenient for our purposes; it is also not hard to show that
it is equivalent to the one in \cite[\S1]{Deligne}, up to the choice of weight.
(Deligne only considers complex variations of weight zero.)

\begin{example} \label{ex:unitary}
Let $\rho \in \Char(X)$ be a unitary character of the fundamental group, and denote
by $\CCrho$ the resulting unitary local system. It determines a polarizable complex
variation of Hodge structure in the following manner. The underlying real local
system is $\RR^2$, with monodromy acting by
\[
	\begin{pmatrix}
		\Re \rho & - \Im \rho \\
		\Im \rho & \Re \rho
	\end{pmatrix};
\]
the standard inner product on $\RR^2$ makes this into a polarizable real variation of
Hodge structure $\varH_{\rho}$ of weight zero, with $J_{\rho} \in \End(\varH_{\rho})$
acting as $J_{\rho}(x,y) = (-y,x)$; for simplicity, we continue to denote the pair
$\bigl( \varH_{\rho}, J_{\rho} \bigr)$ by the symbol $\CCrho$.
\end{example}

We have the following criterion for deciding whether a polarizable complex Hodge
module is \define{smooth}, meaning induced by a complex variation of Hodge structure.

\begin{lemma} \label{lem:CHM-smooth}
Given $(M, J) \in \HMC{X}{w}$, let us denote by
\[
	\Mmod = \Mmod' \oplus \Mmod'' = \ker(J - i \cdot \id) \oplus \ker(J + i \cdot \id)
\]
the induced decomposition of the regular holonomic $\Dmod$-module underlying $M$.
If $\Mmod'$ is coherent as an $\OX$-module, then $M$ is smooth.
\end{lemma}

\begin{proof}
Let $\ratMC = \ker(\ratJC - i \cdot \id) \oplus \ker(\ratJC + i \cdot \id)$ be the
analogous decomposition of the underlying perverse sheaf. Since $\Mmod'$ is
$\OX$-coherent, it is a vector bundle with flat connection; by the Riemann-Hilbert
correspondence, the first factor is therefore (up to a shift in degree by $\dim X$) a
complex local system. Since it is isomorphic to $\ratM$ as a real perverse sheaf, it
follows that $\ratM$ is also a local system; but then $M$ is smooth by
\cite[Lemme~5.1.10]{Saito-HM}.
\end{proof}

In general, the relationship between complex Hodge modules and complex variations of
Hodge structure is governed by the following theorem; it is of course an immediate
consequence of Saito's results (see \theoremref{thm:Saito-structure}).

\begin{theorem} \label{thm:structure-complex}
The category of polarizable complex Hodge modules of weight $w$ with strict support
$Z \subseteq X$ is equivalent to the category of generically defined polarizable
complex variations of Hodge structure of weight $w - \dim Z$ on $Z$.
\end{theorem}

\subsection{Integral structures on Hodge modules}

By working with polarizable real (or complex) Hodge modules, we lose certain
arithmetic information about the monodromy of the underlying perverse sheaves, such
as the fact that the monodromy eigenvalues are roots of unity. One can recover some of
this information by asking for the existence of an ``integral structure''
\cite[Definition~1.9]{Schnell-laz}, which is just a constructible complex of sheaves
of $\ZZ$-modules that becomes isomorphic to the perverse sheaf underlying the Hodge
module after tensoring by $\RR$.

\begin{definition}
An \define{integral structure} on a polarizable real Hodge module $M \in \HM{X}{w}$
is a constructible complex $E \in \Dbc(\ZZ_X)$ such that $\ratM \simeq E \tensor_{\ZZ} \RR$.
\end{definition}

As explained in \cite[\S1.2.2]{Schnell-laz}, the existence of an integral structure
is preserved by many of the standard operations on (mixed) Hodge modules, such as
direct and inverse images or duality. Note that even though it makes sense to ask
whether a given (mixed) Hodge module admits an integral structure, there appears to
be no good functorial theory of ``polarizable integral Hodge modules''.

\begin{lemma} \label{lem:integral-summand}
If $M \in \HM{X}{w}$ admits an integral structure, then the same is true for every
summand in the decomposition of $M$ by strict support.
\end{lemma}

\begin{proof}
Consider the decomposition 
\[
	M = \bigoplus_{j=1}^n M_j
\]
by strict support, with $Z_1, \dotsc, Z_n \subseteq X$ distinct irreducible analytic
subvarieties. Each $M_j$ is a polarizable real Hodge module with strict support
$Z_j$, and therefore comes from a polarizable real variation of Hodge structure
$\varH_j$ on a dense Zariski-open subset of $Z_j$. What we have to prove is that each
$\varH_j$ can be defined over $\ZZ$.  Let $\ratM$ denote the underlying real perverse
sheaf, and set $d_j = \dim Z_j$. According to \cite[Proposition~2.1.17]{BBD}, $Z_j$
is an irreducible component in the support of the $(-d_j)$-th cohomology sheaf of
$\ratM$, and $\varH_{j, \RR}$ is the restriction of that constructible sheaf to a
Zariski-open subset of $Z_j$.  Since $\ratM \simeq E \tensor_{\ZZ} \RR$, it follows
that $\varH_j$ is defined over the integers. 
\end{proof}

\subsection{Operations on Hodge modules}
\label{par:operations}

In this section, we recall three useful operations for polarizable real (and
complex) Hodge modules. If $\Supp M$ is compact, we define the \define{Euler
characteristic} of $M = (\Mmod, F_{\bullet} \Mmod, \ratM) \in \HM{X}{w}$ by the formula
\[
	\chi(X, M) = \sum_{i \in \ZZ} (-1)^i \dim_{\RR} H^i(X, \ratM)
		= \sum_{i \in \ZZ} (-1)^i \dim_{\CC} H^i \bigl( X, \DR(\Mmod) \bigr).
\]
For $(M, J) \in \HMC{X}{w}$, we let $\Mmod = \Mmod' \oplus \Mmod'' = \ker(J - i \cdot
\id) \oplus \ker(J + i \cdot \id)$ be the decomposition into eigenspaces, and define
\[
	\chi(X, M, J) 
		= \sum_{i \in \ZZ} (-1)^i \dim_{\CC} H^i \bigl( X, \DR(\Mmod') \bigr).
\]
With this definition, one has $\chi(X, M) = \chi(X, M, J) + \chi(X, M, -J)$.

Given a smooth morphism $f \colon Y \to X$ of relative dimension $\dim f = \dim Y -
\dim X$, we define the \define{naive inverse image}
\[
	f^{-1} M = \bigl( \fu \Mmod, \fu F_{\bullet} \Mmod, f^{-1} \ratM \bigr).
\]
One can show that $f^{-1} M \in \HM{Y}{w + \dim f}$; see \cite[\S9]{Schnell-van} for
more details. The same is true for polarizable complex Hodge modules: if
$(M, J) \in \HMC{X}{w}$, then one obviously has 
\[
	f^{-1}(M, J) = \bigl( f^{-1} M, f^{-1} J \bigr) \in \HMC{Y}{w + \dim f}.
\]
One can also twist a polarizable complex Hodge module by a unitary character.

\begin{lemma} \label{lem:twist}
For any unitary character $\rho \in \Char(X)$, there is an object
\[
	(M, J) \tensor_{\CC} \CCrho \in \HMC{X}{w}
\]
whose associated complex perverse sheaf is $\ker(\ratJC - i \cdot \id) \tensor_{\CC}
\CCrho$.
\end{lemma}

\begin{proof}
In the notation of \exampleref{ex:unitary}, consider the tensor product
\[
	M \tensor_{\RR} \varH_{\rho} \in \HM{X}{w};
\]
it is again a polarizable real Hodge module of weight $w$ because $\varH_{\rho}$ is a
polarizable real variation of Hodge structure of weight zero. The square of the
endomorphism $J \tensor J_{\rho}$ is the identity, and so
\[
	N = \ker \bigl( J \tensor J_{\rho} + \id \bigr) 
		\subseteq M \tensor_{\RR} \varH_{\rho}
\]
is again a polarizable real Hodge module of weight $w$. Now $K = J \tensor \id
\in \End(N)$ satisfies $K^2 = -\id$, which means that the pair $(N, K)$ is a
polarizable complex Hodge module of weight $w$. On the associated complex perverse
sheaf
\[
	\ker \bigl( K_{\CC} - i \cdot \id \bigr) 
		\subseteq \ratMC \tensor_{\CC} \varH_{\rho, \CC},
\]
both $\ratJC \tensor \id$ and $\id \tensor J_{\rho, \CC}$ act as multiplication by $i$,
which means that
\[
	\ker \bigl( K_{\CC} - i \cdot \id \bigr) 
		= \ker(\ratJC - i \cdot \id) \tensor_{\CC} \CCrho.
\]
The corresponding regular holonomic $\Dmod$-module is obviously
\[
	\Nmod' = \Mmod' \tensor_{\OX} (L, \nabla),
\]
with the filtration induced by $F_{\bullet} \Mmod'$; here $(L, \nabla)$ denotes the
flat bundle corresponding to the complex local system $\CCrho$, and $\Mmod = \Mmod'
\oplus \Mmod''$ as above. 
\end{proof}

\begin{note}
The proof shows that 
\begin{align*}
	N_{\CC} &= \bigl( \ker(\ratJC - i \cdot \id) \tensor_{\CC} \CCrho \bigr) 
		\oplus \bigl( \ker(\ratJC + i \cdot \id) \tensor_{\CC} \CCrhob \bigr) \\
	\Nmod &= \bigl( \Mmod' \tensor_{\OX} (L, \nabla) \bigr) 
		\oplus \bigl( \Mmod'' \tensor_{\OX} (L, \nabla)^{-1} \bigr),
\end{align*}
where $\rhob$ is the complex conjugate of the character $\rho \in \Char(X)$.
\end{note}

\section{Hodge modules on complex tori}

\subsection{Main result}

The paper \cite{PS} contains several results about Hodge modules of geometric origin
on abelian varieties. In this chapter, we generalize these results to arbitrary
polarizable complex Hodge modules on compact complex tori. To do so, we develop a
beautiful idea due to Wang \cite{Wang}, namely that up to direct sums and character
twists, every such object actually comes from an abelian variety.

\begin{theorem} \label{thm:CHM-main}
Let $(M, J) \in \HMC{T}{w}$ be a polarizable complex Hodge module on a compact
complex torus $T$. Then there is a decomposition
\begin{equation} \label{eq:decomposition}
	(M, J) \simeq \bigoplus_{j=1}^n 
		q_j^{-1} (N_j, J_j) \tensor_{\CC} \CC_{\rho_j}
\end{equation}
where $q_j \colon T \to T_j$ is a surjective morphism with connected fibers,
$\rho_j \in \Char(T)$ is a unitary character, and $(N_j, J_j) \in \HMC{T_j}{w -
\dim q_j}$ is a simple polarizable complex Hodge module with $\Supp N_j$ projective
and $\chi(T_j, N_j, J_j) > 0$.
\end{theorem}

For Hodge modules of geometric origin, a less precise result was proved by Wang
\cite{Wang}. His proof makes use of the decomposition theorem, which in the setting
of arbitrary compact K\"ahler manifolds, is only known for Hodge modules of geometric
origin \cite{Saito-Kae}. This technical issue can be circumvented by putting
everything in terms of generically defined variations of Hodge structure. 

To get a result for a polarizable real Hodge module $M \in \HM{T}{w}$, we simply
apply \theoremref{thm:CHM-main} to its complexification $(M \oplus M, J_M) \in
\HMC{T}{w}$. One could say more about the terms in the decomposition below, but the
following version is enough for our purposes.

\begin{corollary} \label{cor:CHM-real}
Let $M \in \HM{T}{w}$ be a polarizable real Hodge module on a compact complex torus
$T$. Then in the notation of \theoremref{thm:CHM-main}, one has
\[
	(M \oplus M, J_M) 
		\simeq \bigoplus_{j=1}^n q_j^{-1}(N_j, J_j) \tensor_{\CC} \CC_{\rho_j}.
\]
If $M$ admits an integral structure, then each $\rho_j \in \Char(T)$ has finite
order.
\end{corollary}

The proof of these results takes up the rest of the chapter.

\subsection{Subvarieties of complex tori}

This section contains a structure theorem for subvarieties of compact complex
tori. The statement is contained in \cite[Propositions~2.3~and~2.4]{Wang}, but 
we give a simpler argument below.

\begin{proposition} \label{prop:structure}
Let $X$ be an irreducible analytic subvariety of a compact complex torus $T$. Then
there is a subtorus $S \subseteq T$ with the following two properties:
\begin{aenumerate}
\item $S + X = X$ and the quotient $Y = X/S$ is projective.
\item If $D \subseteq X$ is an irreducible analytic subvariety with $\dim D = \dim X -
1$, then $S + D = D$.
\end{aenumerate}
In particular, every divisor on $X$ is the preimage of a divisor on $Y$.
\end{proposition}

\begin{proof}
It is well-known that the algebraic reduction of $T$ is an abelian variety. More
precisely, there is a subtorus $S \subseteq T$ such that $A = T/S$ is an abelian
variety, and every other subtorus with this property contains $S$; see e.g.
\cite[Ch.2~\S6]{BL}.

Now let $X \subseteq T$ be an irreducible analytic subvariety of $T$; without loss of
generality, we may assume that $0 \in X$ and that $X$ is not contained in any proper
subtorus of $T$. By a theorem of Ueno \cite[Theorem~10.9]{Ueno}, there is a subtorus
$S' \subseteq T$ with $S' + X \subseteq X$ and such that $X/S' \subseteq T/S'$ is of
general type. In particular, $X/S'$ is projective; but then $T/S'$ must also be
projective, which means that $S \subseteq S'$. Setting $Y = X/S$, we get a cartesian
diagram
\[
\begin{tikzcd}
X \rar[hook] \dar & T \dar \\
Y \rar[hook] & A
\end{tikzcd}
\]
with $Y$ projective. Now it remains to show that every divisor on $X$ is the pullback
of a divisor from $Y$.

Let $D \subseteq X$ be an irreducible analytic subvariety with $\dim D = \dim X - 1$;
as before, we may assume that $0 \in D$. For dimension reasons, either $S + D = D$ or
$S + D = X$; let us suppose that $S + D = X$ and see how this leads to a
contradiction. Define $T_D \subseteq T$ to be the smallest subtorus of $T$ containing
$D$; then $S + T_D = T$. If $T_D = T$, then the same reasoning as above would show that $S
+ D = D$; therefore $T_D \neq T$, and $\dim (T_D \cap S) \leq \dim S - 1$. Now
\[
	D \cap S \subseteq T_D \cap S \subseteq S,
\]
and because $\dim(D \cap S) = \dim S - 1$, it follows that $D \cap S = T_D \cap S$
consists of a subtorus $S''$ and finitely many of its translates. After dividing out
by $S''$, we may assume that $\dim S = 1$ and that $D \cap S = T_D \cap S$ is a
finite set; in particular, $D$ is finite over $Y$, and therefore also projective. Now
consider the addition morphism
\[
	S \times D \to T.
\]
Since $S + D = X$, its image is equal to $X$; because $S$ and $D$ are both
projective, it follows that $X$ is projective, and hence that $T$ is projective. But
this contradicts our choice of $S$. The conclusion is that $S + D = D$, as asserted.
\end{proof}

\begin{note}
It is possible for $S$ to be itself an abelian variety; this is why the proof that $S
+ D \neq X$ requires some care.
\end{note}

\subsection{Simple Hodge modules and abelian varieties}

We begin by proving a structure theorem for \emph{simple} polarizable complex Hodge
modules on a compact complex torus $T$; this is evidently the most important case,
because every polarizable complex Hodge module is isomorphic to a direct sum of
simple ones. Fix a simple polarizable complex Hodge module $(M, J) \in \HMC{T}{w}$.
By \propositionref{prop:semi-simple}, the polarizable real Hodge module $M \in
\HM{X}{w}$ has strict support equal to an irreducible analytic subvariety; we assume
in addition that $\Supp M$ is not contained in any proper subtorus of $T$.

\begin{theorem} \label{thm:CHM}
There is an abelian variety $A$, a surjective morphism $q \colon T \to A$ with
connected fibers, a simple $(N, K) \in \HMC{A}{w - \dim q}$ with $\chi(A, N,
K) > 0$, and a unitary character $\rho \in \Char(T)$, such that 
\begin{equation} \label{eq:CHM}
	(M, J) \simeq q^{-1} (N, K) \tensor_{\CC} \CCrho.
\end{equation}
In particular, $\Supp M = q^{-1}(\Supp N)$ is covered by translates of $\ker q$.
\end{theorem}

Let $X = \Supp M$. By \propositionref{prop:structure}, there is a subtorus $S
\subseteq T$ such that $S + X = X$ and such that $Y = X/S$ is projective. Since $Y$
is not contained in any proper subtorus, it follows that $A = T/S$ is an abelian
variety. Let $q \colon T \to A$ be the quotient mapping, which is proper and smooth of
relative dimension $\dim q = \dim S$. This will not be our final choice for 
\theoremref{thm:CHM}, but it does have almost all the properties that we want (except
for the lower bound on the Euler characteristic).

\begin{proposition} \label{prop:CHM-weak}
There is a simple $(N, K) \in \HMC{A}{w - \dim q}$ with strict support $Y$ and a
unitary character $\rho \in \Char(T)$ for which \eqref{eq:CHM} holds.
\end{proposition}

By \theoremref{thm:structure-complex}, $(M, J)$ corresponds to a polarizable complex
variation of Hodge structure of weight $w - \dim X$ on a dense Zariski-open subset of
$X$. The crucial observation, due to Wang, is that we can choose this set to be of
the form $q^{-1}(U)$, where $U$ is a dense Zariski-open subset of the smooth locus of
$Y$.

\begin{lemma}
There is a dense Zariski-open subset $U \subseteq Y$, contained in the smooth locus
of $Y$, and a polarizable complex variation of Hodge structure $(\varH, J)$ of weight
$w - \dim X$ on $q^{-1}(U)$, such that $(M, J)$ is the polarizable complex Hodge module
corresponding to $(\varH, J)$ in \theoremref{thm:structure-complex}.
\end{lemma}

\begin{proof}
Let $Z \subseteq X$ be the union of the singular locus of $X$ and the singular locus
of $M$. Then $Z$ is an analytic subset of $X$, and according to
\theoremref{thm:Saito-structure}, the restriction of $M$ to $X \setminus Z$ is a
polarizable real variation of Hodge
structure of weight $w - \dim X$. By \propositionref{prop:structure}, no
irreducible component of $Z$ of dimension $\dim X - 1$ dominates $Y$; we can
therefore find a Zariski-open subset $U \subseteq Y$, contained in the smooth locus
of $Y$, such that the intersection $q^{-1}(U) \cap Z$ has codimension $\geq 2$ in
$q^{-1}(U)$. Now $\varH$ extends uniquely to a polarizable real variation of Hodge
structure on the entire complex manifold $q^{-1}(U)$, see
\cite[Proposition~4.1]{Schmid}. The assertion about $J$ follows easily.
\end{proof}

For any $y \in U$, the restriction of $(\varH, J)$ to the fiber $q^{-1}(y)$ is a
polarizable complex variation of Hodge structure on a translate of the compact complex
torus $\ker q$. By \lemmaref{lem:VHS-torus}, the restriction to $q^{-1}(y)$ of the
underlying local system 
\[
	\ker \Bigl( \ratJC - i \cdot \id \, \colon \varHC \to \varHC \Bigr)
\]
is the direct sum of local systems of the form $\CCrho$, for $\rho \in \Char(T)$
unitary; when $M$ admits an integral structure, $\rho$ has finite order in the group
$\Char(T)$.

\begin{proof}[Proof of \propositionref{prop:CHM-weak}]
Let $\rho \in \Char(T)$ be one of the unitary characters in question, and let 
$\rhob \in \Char(T)$ denote its complex conjugate. The tensor product $(\varH, J)
\tensor_{\CC} \CCrhob$ is a polarizable complex variation of Hodge structure of
weight $w - \dim X$ on the open subset $q^{-1}(U)$. Since all fibers of $q \colon
q^{-1}(U) \to U$ are translates of the compact complex torus $\ker q$, classical
Hodge theory for compact K\"ahler manifolds \cite[Theorem~2.9]{Zucker} implies that
\begin{equation} \label{eq:ql-varH}
	\ql \bigl( (\varH, J) \tensor_{\CC} \CCrhob \bigr)
\end{equation}
is a polarizable complex variation of Hodge structure of weight $w - \dim X$ on
$U$; in particular, it is again semi-simple. By our choice of $\rho$, the adjunction
morphism
\[
	q^{-1} \ql \bigl( (\varH, J) \tensor_{\CC} \CCrhob \bigr) \to 
		(\varH, J) \tensor_{\CC} \CCrhob
\]
is nontrivial. Consequently, \eqref{eq:ql-varH} must have at least one simple summand
$(\varH_U, K)$ in the category of polarizable complex variations of Hodge structure of
weight $w - \dim X$ for which the induced morphism $q^{-1} (\varH_U, K) \to (\varH,
J) \tensor_{\CC} \CCrhob$ is nontrivial. Both sides being simple, the 
morphism is an isomorphism; consequently, 
\begin{equation} \label{eq:varHY}
	q^{-1}(\varH_U, K) \tensor_{\CC} \CCrho \simeq (\varH, J).
\end{equation}
Now let $(N, K) \in \HMC{A}{w - \dim q}$ be the polarizable complex Hodge module on
$A$ corresponding to $(\varH_U, K)$; by construction, $(N, K)$ is simple with strict
support $Y$.  Arguing as in \cite[Lemma~20.2]{Schnell-holo}, one proves that the naive
pullback $q^{-1}(N, K) \in \HMC{T}{w}$ is simple with strict support $X$.  Because of
\eqref{eq:varHY}, this means that $(M, J)$ is isomorphic to $q^{-1}(N, K)
\tensor_{\CC} \CCrho$ in the category $\HMC{T}{w}$.
\end{proof}

We have thus proved \theoremref{thm:CHM}, except for the inequality $\chi(A, N, K) >
0$. Let $\Nmod$ denote the regular holonomic $\Dmod$-module underlying $N$; then
\[
	\Nmod = \Nmod' \oplus \Nmod'' 
		= \ker(K - i \cdot \id) \oplus \ker(K + i \cdot \id),
\]
where $K \in \End(\Nmod)$ refers to the induced endomorphism. By
\propositionref{prop:CHM-Kaehler}, both $\Nmod'$ and $\Nmod''$ are simple 
with strict support $Y$. Since $A$ is an abelian variety, one has for example by
\cite[\S5]{Schnell-holo} that
\[
	\chi(A, N, K) = \sum_{i \in \ZZ} (-1)^i \dim H^i \bigl( A, \DR(\Nmod') \bigr) 
		\geq 0.
\]
Now the point is that a simple holonomic $\Dmod$-module with vanishing Euler
characteristic is always (up to a twist by a line bundle with flat connection) the
pullback from a lower-dimensional abelian variety \cite[\S20]{Schnell-holo}.

\begin{proof}[Proof of \theoremref{thm:CHM}]
Keeping the notation from \propositionref{prop:CHM-weak}, we have a surjective
morphism $q \colon T \to A$ with connected fibers, a simple polarizable complex Hodge
module $(N, K) \in \HMC{Y}{w - \dim q}$ with strict support $Y = q(X)$, and a unitary
character $\rho \in \Char(T)$ such that
\[
	(M, J) \simeq q^{-1}(N, K) \tensor_{\CC} \CCrho.
\]
If $(N, K)$ has positive Euler characteristic, we are done, so let us assume from now
on that $\chi(A, N, K) = 0$. This means that $\Nmod'$ is a simple regular holonomic
$\Dmod$-module with strict support $Y$ and Euler characteristic zero.

By \cite[Corollary~5.2]{Schnell-holo}, there is a surjective morphism $f \colon A \to
B$ with connected fibers from $A$ to a lower-dimensional abelian variety $B$, such
that $\Nmod'$ is (up to a twist by a line bundle with flat connection) the pullback of
a simple regular holonomic $\Dmod$-module with positive Euler characteristic. Setting
\[
	\Mmod = \Mmod' \oplus \Mmod'' 
		= \ker(J - i \cdot \id) \oplus \ker(J + i \cdot \id),
\]
it follows that $\Mmod'$ is (again up to a twist by a line bundle with flat
connection) the pullback by $f \circ q$ of a simple regular holonomic $\Dmod$-module
on $B$.  Consequently, there is a dense Zariski-open subset $U \subseteq f(Y)$ such that the
restriction of $\Mmod'$ to $(f \circ q)^{-1}(U)$ is coherent as an $\shO$-module. By
\lemmaref{lem:CHM-smooth}, the restriction of $(M, J)$ to this open set is therefore
a polarizable complex variation of Hodge structure of weight $w - \dim X$. After
replacing our original morphism $q \colon T \to A$ by the composition $f \circ q
\colon T \to B$, we can argue as in the proof of \propositionref{prop:CHM-weak} to
show that \eqref{eq:CHM} is still satisfied (for a different choice of $\rho \in
\Char(T)$, perhaps).  

With some additional work, one can prove that now $\chi(A, N, K) > 0$. Alternatively,
the same result can be obtained by the following more indirect method: as long as
$\chi(A, N, K) = 0$, we can repeat the argument above; since the dimension of $A$
goes down each time, we must eventually get to the point where $\chi(A, N, K) > 0$. This
completes the proof of \theoremref{thm:CHM}.
\end{proof}

\subsection{Proof of the main result}

As in \theoremref{thm:CHM-main}, let $(M, J) \in \HMC{T}{w}$ be a polarizable complex
Hodge module on a compact complex torus $T$. Using the decomposition by strict
support, we can assume without loss of generality that $(M, J)$ has strict support
equal to an irreducible analytic subvariety $X \subseteq T$.  After translation, we
may assume moreover that $0 \in X$. Let $T' \subseteq T$ be the smallest subtorus of
$T$ containing $X$; by Kashiwara's equivalence, we have $(M, J) = \il (M', J')$ for
some $(M', J') \in \HMC{T'}{w}$, where $i \colon T' \into T$ is the inclusion. Now
\theoremref{thm:CHM} gives us a morphism $q' \colon T' \to A'$ such that $(M', J')$
is isomorphic to the direct sum of pullbacks of polarizable complex Hodge modules
twisted by unitary local systems. Since $i^{-1} \colon \Char(T) \to \Char(T')$ is
surjective, the same is then true for $(M, J)$ with respect to the quotient mapping
$q \colon T \to T/\ker q'$. This proves \theoremref{thm:CHM-main}.

\begin{proof}[Proof of \corollaryref{cor:CHM-real}]
By considering the complexification 
\[
	(M \oplus M, J_M) \in \HMC{T}{w}, 
\]
we reduce the problem to the situation considered in \theoremref{thm:CHM-main}. 
It remains to show that all the characters in \eqref{eq:decomposition} have finite
order in $\Char(T)$ if $M$ admits an integral structure. By \lemmaref{lem:integral-summand}, every summand in the
decomposition of $M$ by strict support still admits an integral structure, and so we
may assume without loss of generality that $M$ has strict support equal to $X
\subseteq T$ and that $0 \in X$. As before, we have $(M, J) = \il (M', J')$, where $i
\colon T' \into T$ is the smallest subtorus of $T$ containing $X$; it is easy to see
that $M'$ again admits an integral structure. Now we apply the same argument as in
the proof of \theoremref{thm:CHM-main} to the finitely many simple factors of $(M,
J)$, noting that the characters $\rho \in \Char(T)$ that come up always have finite
order by \lemmaref{lem:VHS-torus} below.
\end{proof}

\begin{note}
As in the proof of \lemmaref{lem:twist}, it follows that $M \oplus M$ is isomorphic
to the direct sum of the polarizable real Hodge modules
\begin{equation} \label{eq:summand-j}
	\ker \Bigl( q_j^{-1} J_j \tensor J_{\rho_j} + \id \Bigr)
		\subseteq q_j^{-1} N_j \tensor_{\RR} \varH_{\rho_j}.
\end{equation}
Furthermore, one can show that for each $j = 1, \dotsc, n$, exactly one of two things
happens. (1) Either the object in \eqref{eq:summand-j} is simple, and therefore
occurs among the simple factors of $M$; in this case, the underlying regular
holonomic $\Dmod$-module $\Mmod$ will contain the two simple factors
\[
	\Bigl( \qu_j \Nmod_j' \tensor_{\OT} (L_j, \nabla_j) \Bigr)
		\oplus \Bigl( \qu_j \Nmod_j'' \tensor_{\OT} (L_j, \nabla_j)^{-1} \Bigr).
\]
(2) Or the object in \eqref{eq:summand-j} splits into two copies of a simple
polarizable real Hodge module, which also has to occur among the simple factors of $M$.
In this case, one can actually arrange that $(N_j, J_j)$ is real and that the
character $\rho_j$ takes values in $\{-1, +1\}$. The simple object in question is the
twist of $(N_j, J_j)$ by the polarizable real variation of Hodge structure of rank
one determined by $\rho_j$; moreover, $\Mmod$ will contain $\qu_j
\Nmod_j' \tensor_{\OT} (L_j, \nabla_j) \simeq \qu_j \Nmod_j'' \tensor_{\OT} (L_j,
\nabla_j)^{-1}$ as a simple factor.
\end{note}

\subsection{A lemma about variations of Hodge structure}

The fundamental group of a compact complex torus is abelian, and so every polarizable
complex variation of Hodge structure is a direct sum of unitary local systems of rank
one; this is the content of the following elementary lemma \cite[Lemma~1.8]{Schnell-laz}.

\begin{lemma} \label{lem:VHS-torus}
Let $(\varH, J)$ be a polarizable complex variation of Hodge structure on a compact complex
torus $T$. Then the local system $\varHC = \varHR \tensor_{\RR} \CC$ is isomorphic to
a direct sum of unitary local systems of rank one. If $\varH$ admits an integral
structure, then each of these local systems of rank one has finite order.
\end{lemma}

\begin{proof}
According to \cite[\S1.12]{Deligne}, the underlying local system of a polarizable
complex variation of Hodge structure on a compact K\"ahler manifold is semi-simple; in
the case of a compact complex torus, it is therefore a direct sum of rank-one local
systems. The existence of a polarization implies that the individual local systems
are unitary \cite[Proposition~1.13]{Deligne}. Now suppose that $\varH$ admits an
integral structure, and let $\mu \colon \pi_1(A, 0) \to \GL_n(\ZZ)$ be the monodromy
representation. We already know that the complexification of $\mu$ is a 
direct sum of unitary characters. Since $\mu$ is defined over $\ZZ$, the values of
each character are algebraic integers of absolute value one; by Kronecker's theorem,
they must be roots of unity.
\end{proof}

\subsection{Integral structure and points of finite order}

One can combine the decomposition in \corollaryref{cor:CHM-real} with known results
about Hodge modules on abelian varieties \cite{Schnell-laz} to prove the following
generalization of Wang's theorem.

\begin{corollary} \label{cor:finite-order}
If $M \in \HM{T}{w}$ admits an integral structure, then the sets
\[
	S_m^i(T, M) 
		= \menge{\rho \in \Char(T)}{\dim H^i(T, \ratM \tensor_{\RR} \CCrho) \geq m}
\]
are finite unions of translates of linear subvarieties by points of finite order.
\end{corollary}

\begin{proof}
The result in question is known for abelian varieties: if $M \in \HM{A}{w}$
is a polarizable real Hodge module on an abelian variety, and if $M$ admits an
integral structure, then the sets $S_m^i(A, M)$
are finite unions of ``arithmetic subvarieties'' (= translates of linear subvarieties
by points of finite order). This is proved in \cite[Theorem~1.4]{Schnell-laz} for
polarizable rational Hodge modules, but the proof carries over unchanged to the case
of real coefficients. The same argument shows more generally that if the underlying
perverse sheaf $\ratMC$ of a polarizable real Hodge module $M \in \HM{A}{w}$ is
isomorphic to a direct factor in the complexification of some $E \in \Dbc(\ZZ_A)$,
then each $S_m^i(A, M)$ is a finite union of arithmetic subvarieties.

Now let us see how to extend this result to compact complex tori. Passing to the
underlying complex perverse sheaves in \corollaryref{cor:CHM-real}, we get
\[
	\ratMC \simeq \bigoplus_{j=1}^n \bigl( q_j^{-1} N_{j, \CC} \tensor_{\CC}
\CC_{\rho_j} \bigr);
\]
recall that $\Supp N_j$ is a projective subvariety of the complex torus $T_j$, and
that $\rho_j \in \Char(T)$ has finite order. In light of this decomposition and the
comments above, it is therefore enough to prove that each $N_{j, \CC}$ is isomorphic
to a direct factor in the complexification of some object of $\Dbc(\ZZ_{T_j})$. 

Let $E \in \Dbc(\ZZ_T)$ be some choice of integral structure on the real Hodge module
$M$; obviously $\ratMC \simeq E \tensor_{\ZZ} \CC$. Let $r \geq 1$ be the order of
the point $\rho_j \in \Char(T)$, and denote by $[r] \colon T \to T$ the finite
morphism given by multiplication by $r$. We define
\[
	E' = \derR [r]_{\ast} \bigl( [r]^{-1} E \bigr) \in \Dbc(\ZZ_T)
\]
and observe that the complexification of $E'$ is isomorphic to the direct sum of $E
\tensor_{\ZZ} \CCrho$, where $\rho \in \Char(T)$ runs over the finite set of
characters whose order divides $r$. This set includes $\rho_j^{-1}$, and so $q_j^{-1}
N_{j, \CC}$ is isomorphic to a direct factor of $E' \tensor_{\ZZ} \CC$. Because $q_j
\colon T \to T_j$ has connected fibers, this implies that
\[
	N_{j, \CC} \simeq \shH^{-\dim q_j} q_{j \ast} \bigl( q_j^{-1} N_{j, \CC} \bigr)
\]
is isomorphic to a direct factor of 
\[
	\shH^{-\dim q_j} q_{j \ast} \bigl( E' \tensor_{\ZZ} \CC \bigr).
\]
As explained in \cite[\S1.2.2]{Schnell-laz}, this is again the complexification of a
constructible complex in $\Dbc(\ZZ_{T_j})$, and so the proof is complete.
\end{proof}

\section{Generic vanishing theory}

Let $X$ be a compact K\"ahler manifold, and let $f \colon X \to T$ be a holomorphic
mapping to a compact complex torus. The main purpose of this chapter is to show that
the higher direct image sheaves $R^j \fl \omX$ have the same properties as in the
projective case (such as being GV-sheaves). As explained in the introduction, we do
not know how to obtain this using classical Hodge theory; this forces us to prove a
more general result for arbitrary polarizable complex Hodge modules.

\subsection{GV-sheaves and M-regular sheaves}
\label{par:GV-sheaves}

We begin by reviewing a few basic definitions. Let $T$ be a compact complex torus,
$\widehat{T} = \Pic^0(T)$ its dual, and $P$ the normalized Poincar\'e
bundle on the product $T \times \widehat T$. It induces an integral transform 
\[
	\derR \Phi_P \colon \Dbcoh(\OT) \to \Dbcoh(\shO_{\widehat T}), \quad
		\derR \Phi_P (\shF) = \derR {p_2}_*(p_1^* \shF \otimes P),
\] 
where $\Dbcoh(\OT)$ is the derived category of cohomologically bounded and coherent
complexes of $\OT$-modules. Likewise, we have $\derR \Psi_P \colon
\Dbcoh(\shO_{\widehat T}) \to \Dbcoh(\shO_T)$ going in the opposite direction.
An argument analogous to Mukai's for abelian varieties shows that the Fourier-Mukai
equivalence holds in this case as well \cite[Theorem~2.1]{BBP}.

\begin{theorem}\label{mukai}
With the notations above, $\derR \Phi_P$ and $\derR \Psi_P$ are equivalences of
derived categories. More precisely, one has 
\[
	\derR \Psi_P \circ \derR \Phi_P \simeq (-1)_T^* [-\dim T] \quad \text{and} \quad
 \derR \Phi_P \circ \derR \Psi_P \simeq (-1)_{\widehat T}^* [-\dim T].
\]
\end{theorem}

%More generally, let $X$ be a compact complex manifold, and $f: X \rightarrow A$ its Albanese map. A normalized Poincar\'e line bundle $P_X$ on $X \times \hat A$ defines similarly an integral transform 
%$$\derR \Phi_{P_X} : \D (X) \longrightarrow \D ( \hat A).$$
%It is not hard to see that the two integral transforms are related by the formula
%$$\derR \Phi_{P_X} \simeq \derR \Phi_P \circ \derR f_*.$$

Given a coherent $\OT$-module $\shF$ and an integer $m\ge 1$, we define 
\[
	S^i_m(T, \shF) = \menge{L \in \Pic^0(T)}{\dim ~H^i(T, \shF \tensor_{\OT} L) \ge m}.
\]
It is customary to denote 
\[
	S^i (T, \shF) = S^i_1(T, \shF) = \menge{L \in \Pic^0(T)}{H^i(T, \shF \tensor_{\OT} L) \neq 0}.
\]
Recall the following definitions from \cite{PP3} and \cite{PP4} respectively.

\begin{definition}
A coherent $\OT$-module $\shF$ is called a \define{GV-sheaf} if the inequality
\[
	\codim_{\Pic^0(T)} S^i(T, \shF) \geq i
\]
is satisfied for every integer $i \geq 0$. It is called \define{M-regular} if the
inequality
\[
	\codim_{\Pic^0(T)} S^i(T, \shF) \geq i+1
\]
is satisfied for every integer $i \geq 1$.
\end{definition}

%Just as in the case of abelian varieties, $\shF$ is a GV-sheaf if and only if its
%Fourier-Mukai transform satisfies
%\[
%	\derR \Phi_P(\shF) \simeq \derR \shHom \bigl( \shFh, \shO_{\Pic^0(T)} \bigr)
%\]
%for a coherent $\shO_{\Pic^0(T)}$-module $\shFh$; then M-regularity is equivalent to
%saying that $\shFh$ is torsion-free.

A number of local properties of integral transforms for complex manifolds, based only
on commutative algebra results, were proved in \cite{PP1,Popa}. For
instance, the following is a special case of \cite[Theorem~2.2]{PP1}.

\begin{theorem}\label{GV_WIT}
Let $\shF$ be a coherent sheaf on a compact complex torus $T$.
Then the following statements are equivalent:
\begin{renumerate}
\item $\shF$ is a GV-sheaf.
\item $R^i \Phi_P (\derR \Delta \shF) = 0$ for $i \neq \dim T$, where $\derR \Delta \shF : = 
\derR \shHom (\shF, \shO_T)$.
\end{renumerate}
\end{theorem}

Note that this statement was inspired by work of Hacon \cite{Hacon} in the projective
setting.  In the course of the proof of \theoremref{GV_WIT}, and also for some of
the results below, the following consequence of Grothendieck duality for compact
complex manifolds is needed; see the proof of \cite[Theorem~2.2]{PP1}, and especially
the references there. 
\begin{equation}\label{duality}
	\derR \Phi_P (\shF) \simeq 
		\derR \Delta \bigl( \derR \Phi_{P^{-1}} (\derR \Delta \shF) [\dim T] \bigr). 
\end{equation}

In particular, if $\shF$ is a GV-sheaf, 
then if we denote $\hat{\shF} : = R^{\dim T} \Phi_{P^{-1}} (\derR \Delta \shF)$,
\theoremref{GV_WIT} and \eqref{duality} imply that
\begin{equation}\label{basic}
 \derR \Phi_P (\shF) \simeq \derR \shHom(\hat{\shF}, \shO_{\hat A}).
 \end{equation}
As in \cite[Proposition~2.8]{PP2}, $\shF$ is $M$-regular if and only if
$\hat{\shF}$ is torsion-free.  

The fact that \theoremref{mukai}, \theoremref{GV_WIT} and ($\ref{basic}$) hold for
arbitrary compact complex tori allows us to deduce important properties of
GV-sheaves in this setting. Besides these statements, the proofs only rely on local
commutative algebra and base change, and so are completely analogous to those for
abelian varieties; we will thus only indicate references for that case. 

\begin{proposition} \label{sliding}
Let $\shF$ be a GV-sheaf on $T$. 
\begin{aenumerate}
\item One has $S^{\dim T} (T, \shF) \subseteq \cdots \subseteq S^1 (T, \shF)
\subseteq S^0 (T, \shF) \subseteq \widehat T$.
\item If $S^0(T, \shF)$ is empty, then $\shF = 0$.
\item If an irreducible component $Z \subseteq S^0 (T, \shF)$ has codimension $k$ in
$\Pic^0(X)$, then $Z \subseteq S^k (T, \shF)$, and hence $\dim \Supp \shF \geq k$.
\end{aenumerate}
\end{proposition}

\begin{proof}
For (a), see \cite[Proposition~3.14]{PP3}; for (b), see \cite[Lemma~1.12]{Pareschi};
for (c), see \cite[Lemma~1.8]{Pareschi}.
\end{proof}

%\begin{lemma}\label{vanishing-low}
%If $\shF$ is a GV-sheaf on $T$, then
%$$R^i \Phi_P (\shF) = 0 \,\,\,\,{\rm for ~all} \,\,\,\, i < {\rm codim}~S^0 (T, \shF).$$
%\end{lemma}

\subsection{Higher direct images of dualizing sheaves}

Saito \cite{Saito-Kae} and Takegoshi \cite{Takegoshi} have extended to K\"ahler
manifolds many of the fundamental theorems on higher direct images of canonical
bundles proved by Koll\'ar for smooth projective varieties. The following theorem
summarizes some of the results in \cite[p.390--391]{Takegoshi} in the special case
that is needed for our purposes.

\begin{theorem}[Takegoshi] \label{takegoshi}
Let $f \colon X \to Y$ be a proper holomorphic mapping from a compact K\"ahler
manifold to a reduced and irreducible analytic space, and let $L \in \Pic^0(X)$ be a
holomorphic line bundle with trivial first Chern class. 
\begin{aenumerate}
\item The Leray spectral sequence
\[
	E^{p,q}_2 = H^p \bigl( Y, R^q \fl(\omX \tensor L) \bigr) \Longrightarrow
		 H^{p+q}(X, \omega_X \tensor L)
\]
degenerates at $E_2$.
\item If $f$ is surjective, then $R^q \fl(\omX \tensor L)$ is torsion free for every
$q\ge 0$; in particular, it vanishes for $q > \dim X - \dim Y$. 
\end{aenumerate}
\end{theorem}

Saito \cite{Saito-Kae} obtained the same results in much greater generality, using the
theory of Hodge modules. In fact, his method also gives the splitting of the complex
$\derR \fl \omX$ in the derived category, thus extending the main result of
\cite{Kollar} to all compact K\"ahler manifolds.

\begin{theorem}[Saito]
Keeping the assumptions of the previous theorem, one has
\[
	\derR \fl \omX \simeq \bigoplus_j \bigl( R^j \fl \omX \bigr) \decal{-j}
\]
in the derived category $\Dbcoh(\OY)$.
\end{theorem}

\begin{proof}
Given \cite{Saito-Kae}, the proof in \cite{Saito-K} goes through under the assumption
that $X$ is a compact K\"ahler manifold.
\end{proof}

\subsection{Euler characteristic and M-regularity}

In this section, we relate the Euler characteristic of a simple polarizable complex
Hodge module on a compact complex torus $T$ to the M-regularity of the associated
graded object.

\begin{lemma} \label{lem:M-regular}
Let $(M, J) \in \HMC{T}{w}$ be a simple polarizable complex Hodge module on a compact
complex torus. If $\Supp M$ is projective and $\chi(T, M, J) > 0$, then the coherent
$\OT$-module $\gr_k^F \Mmod'$ is M-regular for every $k \in \ZZ$.
\end{lemma}

\begin{proof}
$\Supp M$ is projective, hence contained in a translate of an abelian subvariety $A
\subseteq T$; because \lemmaref{lem:Kashiwara} holds for polarizable complex Hodge
modules, we may therefore assume without loss of generality that $T = A$ is an
abelian variety.

As usual, let $\Mmod = \Mmod' \oplus \Mmod'' = \ker(J - i \cdot \id) \oplus \ker(J +
i \cdot \id)$ be the decomposition into eigenspaces. The summand $\Mmod'$ is a simple
holonomic $\Dmod$-module with positive Euler characteristic on an abelian variety,
and so \cite[Theorem~2.2 and Corollary~20.5]{Schnell-holo} show that
\begin{equation} \label{eq:CSL}
	\menge{\rho \in \Char(A)}%
		{H^i \bigl( A, \DR(\Mmod') \tensor_{\CC} \CCrho \bigr) \neq 0}
\end{equation}
is equal to $\Char(A)$ when $i = 0$, and is equal to a finite union of translates of
linear subvarieties of codimension $\geq 2i+2$ when $i \geq 1$. 

We have a one-to-one correspondence between $\Pic^0(A)$ and the subgroup of unitary
characters in $\Char(A)$; it takes a unitary character $\rho \in \Char(A)$ to the
holomorphic line bundle $L_{\rho} = \CCrho \tensor_{\CC} \OA$. If $\rho \in \Char(A)$ is
unitary, the twist $(M, J) \tensor_{\CC} \CCrho$ is still a polarizable complex Hodge
module by \lemmaref{lem:twist}, and so the complex computing its hypercohomology is
strict. It follows that 
\[
	H^i \bigl( A, \gr_k^F \DR(\Mmod') \tensor_{\OA} L_{\rho} \bigr)
		\quad \text{is a subquotient of} \quad
		H^i \bigl( A, \DR(\Mmod') \tensor_{\CC} \CCrho \bigr).
\]
If we identify $\Pic^0(A)$ with the subgroup of unitary characters, this means that
\[
	\menge{L \in \Pic^0(A)}%
		{H^i \bigl( A, \gr_k^F \DR(\Mmod') \tensor_{\OA} L \bigr) \neq 0}
\]
is contained in the intersection of \eqref{eq:CSL} and the subgroup of unitary
characters. When $i \geq 1$, this intersection is a finite union of translates of
subtori of codimension $\geq i+1$; it follows that
\[
	\codim_{\Pic^0(A)} \menge{L \in \Pic^0(A)}%
		{H^i \bigl( A, \gr_k^F \DR(\Mmod') \tensor_{\OA} L \bigr) \neq 0} \geq i+1.
\]
Since the cotangent bundle of $A$ is trivial, a simple induction on $k$ as in the
proof of \cite[Lemma~1]{PS} gives
\[
	\codim_{\Pic^0(A)} \menge{L \in \Pic^0(A)}
		{H^i \bigl( A, \gr_k^F \Mmod' \tensor_{\OA} L \bigr) \neq 0} \geq i+1,
\]
and so each $\gr_k^F \Mmod'$ is indeed M-regular.
\end{proof}

\begin{note}
In fact, the result still holds without the assumption that $\Supp M$ is projective;
this is an easy consequence of the decomposition in \eqref{eq:decomposition}.
\end{note}

\subsection{Chen-Jiang decomposition and generic vanishing}
\label{par:ChenJiang}

Using the decomposition in \theoremref{thm:CHM-main} and the result of the previous
section, we can now prove the most general version of the generic vanishing theorem,
namely \theoremref{thm:Chen-Jiang} in the introduction.

\begin{proof}[Proof of  \theoremref{thm:Chen-Jiang}]
We apply \theoremref{thm:CHM-main} to the complexification $(M \oplus M, J_M) \in
\HMC{T}{w}$. Passing to the associated graded in \eqref{eq:decomposition}, we obtain
a decomposition of the desired type with $\shF_j = \gr_k^F \Nmod_j'$ and $L_j =
\CC_{\rho_j} \tensor_{\CC} \OT$, where
\[
	\Nmod_j = \Nmod_j' \oplus \Nmod_j'' = \ker(J_j - i \cdot \id) \oplus
		\ker(J_j + i \cdot \id)
\]
is as usual the decomposition into eigenspaces of $J_j \in \End(\Nmod_j)$. Since
$\Supp N_j$ is projective and $\chi(T_j, N_j, J_j) > 0$, we conclude from
\lemmaref{lem:M-regular} that each coherent $\shO_{T_j}$-module $\shF_j$ is M-regular.
\end{proof}

\begin{corollary}\label{cor:MHM-GV}
If $M = (\Mmod, F_{\bullet} \Mmod, \ratM) \in \HM{T}{w}$, then for every $k \in \ZZ$,
the coherent $\OT$-module $\gr_k^F \Mmod$ is a GV-sheaf.
\end{corollary}

\begin{proof}
This follows immediately from \theoremref{thm:Chen-Jiang} and the fact that if $p
\colon T \to T_0$ is a surjective homomorphism of complex tori and $\shG$ is a GV-sheaf on
$T_0$, then $\shF = \fu \shG$ is a GV-sheaf on $T$. For this last statement and more
refined facts (for instance when $\shG$ is $M$-regular), see e.g.
\cite[\S2]{ChenJiang}, especially Proposition 2.6.  The arguments in \cite{ChenJiang}
are for abelian varieties, but given the remarks in \parref{par:GV-sheaves}, they
work equally well on compact complex tori.
\end{proof}

By specializing to the direct image of the canonical Hodge module $\RR_X \decal{\dim
X}$ along a morphism $f \colon X \to T$, we are finally able to conclude that each
$R^j \fl \omX$ is a GV-sheaf. In fact, we have the more refined \theoremref{thm:direct_image}; 
it was first proved for smooth projective varieties of maximal Albanese dimension by Chen
and Jiang \cite[Theorem~1.2]{ChenJiang}, which was a source of inspiration for us. 

\begin{proof}[Proof of \theoremref{thm:direct_image}]
Denote by $\RR_X \decal{\dim X} \in \HM{X}{\dim X}$ the polarizable real Hodge module
corresponding to the constant real variation of Hodge structure of rank one and
weight zero on $X$.  According to \cite[Theorem~3.1]{Saito-Kae}, each $\shH^j
\fl \RR_X \decal{\dim X}$ is a polarizable real Hodge module of weight $\dim X + j$
on $T$; it also admits an integral structure \cite[\S1.2.2]{Schnell-laz}. In the
decomposition by strict support, let $M$ be the summand with strict support $f(X)$;
note that $M$
still admits an integral structure by \lemmaref{lem:integral-summand}. Now $R^j \fl
\omX$ is the first nontrivial piece of the Hodge filtration on the underlying regular
holonomic $\Dmod$-module \cite{Saito-K}, and so the result follows directly from
\theoremref{thm:Chen-Jiang} and \corollaryref{cor:MHM-GV}. For the ampleness 
part in the statement, see \corollaryref{cor:MHM_positivity}.
\end{proof}

\begin{note}
Except for the assertion about finite order, \theoremref{thm:direct_image}  still holds
 for arbitrary coherent $\OT$-modules of the form 
\[
	R^j \fl ( \omX\tensor L)
\]
with $L \in \Pic^0(X)$. The point is that every such $L$ is the holomorphic line
bundle associated with a unitary character $\rho \in \Char(X)$; we can therefore
apply the same argument as above to the polarizable complex Hodge module $\CCrho
\decal{\dim X}$.
\end{note}

If the given morphism is generically finite over its image, we can say more. 

\begin{corollary} \label{cor:gen-finite}
If $f \colon X \to T$ is generically finite over its image, then $S^0(T,
\fl \omX)$ is preserved by the involution $L \mapsto L^{-1}$ of $\Pic^0(T)$.
\end{corollary}

\begin{proof}
As before, we define $M = \shH^0 \fl \RR_X \decal{\dim X} \in \HM{T}{\dim X}$.
Recall from \corollaryref{cor:CHM-real} that we have a decomposition
\[
	(M \oplus M, J_M) \simeq \bigoplus_{j=1}^n \bigl( q_j^{-1}(N_j, J_j) 
		\tensor_{\CC} \CC_{\rho_j} \bigr).
\]
Since $f$ is generically finite over its image, there is a dense Zariski-open subset
of $f(X)$ where $M$ is a variation of Hodge structure of type $(0,0)$; the
above decomposition shows that the same is true for $N_j$ on $(q_j \circ f)(X)$. If
we pass to the underlying regular holonomic $\Dmod$-modules and remember
\lemmaref{lem:twist}, we see that
\[
	\Mmod \oplus \Mmod \simeq \bigoplus_{j=1}^n 
		\Bigl( \qu_j \Nmod_j' \tensor_{\OT} (L_j, \nabla_j) \Bigr) \oplus
		\bigoplus_{j=1}^n \Bigl( \qu_j \Nmod_j'' \tensor_{\OT} (L_j, \nabla_j)^{-1} \Bigr),
\]
where $(L_j, \nabla_j)$ is the flat bundle corresponding to the character $\rho_j$.
By looking at the first nontrivial step in the Hodge filtration on $\Mmod$, we then get
\[
	\fl \omX \oplus \fl \omX \simeq 
		\bigoplus_{j=1}^n \Bigl( \qu_j \shF_j' \tensor_{\OT} L_j \Bigr)
		\oplus \bigoplus_{j=1}^n \Bigl( \qu_j \shF_j'' \tensor_{\OT} L_j^{-1} \Bigr),
\]
where $\shF_j' = F_{p(M)} \Nmod_j'$ and $\shF_j'' = F_{p(M)} \Nmod_j''$, and $p(M)$
is the smallest integer with the property that $F_p \Mmod \neq 0$. Both sheaves are
torsion-free on $(q_j \circ f)(X)$, and can therefore be nonzero only when $\Supp N_j
= (q_j \circ f)(X)$; after re-indexing, we may assume that this holds exactly in the
range $1 \leq j \leq m$. 

Now we reach the crucial point of the argument: the fact that $N_j$ is generically a
polarizable real variation of Hodge structure of type $(0,0)$ implies that $\shF_j'$
and $\shF_j''$ have the same rank at the generic point of $(q_j \circ f)(X)$. Indeed,
on a dense Zariski-open subset of $(q_j \circ f)(X)$, we have $\shF_j' = \Nmod_j'$
and $\shF_j'' = \Nmod_j''$, and complex conjugation with respect to the real
structure on $N_j$ interchanges the two factors. 

Since $\shF_j'$ and $\shF_j''$ are M-regular by \lemmaref{lem:M-regular}, we have (for
$1 \leq j \leq m$)
\[
	S^0(T, \qu_j \shF_j' \tensor_{\OT} L_j) =
		L_j^{-1} \tensor S^0(T_j, \shF_j') 
		= L_j^{-1} \tensor \Pic^0(T_j),
\]
and similarly for $\qu_j \shF_j'' \tensor_{\OT} L_j^{-1}$; to simplify the notation,
we identify $\Pic^0(T_j)$ with its image in $\Pic^0(T)$. The decomposition from above
now gives
\[
	S^0(T, \fl \omX) = \bigcup_{j=1}^m \Bigl( L_j^{-1} \tensor \Pic^0(T_j) \Bigr)
		\cup \bigcup_{j=1}^m \Bigl( L_j \tensor \Pic^0(T_j) \Bigr),
\]
and the right-hand side is clearly preserved by the involution $L \mapsto L^{-1}$.
\end{proof}

\subsection{Points of finite order on cohomology support loci}

Let $f \colon X \to T$ be a holomorphic mapping from a compact K\"ahler manifold to a
compact complex torus. Our goal in this section is to prove that the cohomology
support loci of the coherent $\OT$-modules $R^j \fl \omX$ are finite unions of
translates of subtori by points of finite order. We consider the refined cohomology
support loci
\[
	S_m^i(T, R^j \fl \omX) = \menge{L \in \Pic^0(T)}%
		{\dim H^i(T, R^j \fl \omX \tensor L) \geq m} \subseteq \Pic^0(T).
\]
The following result is well-known in the projective case.

\begin{corollary} \label{cor:Wang}
Every irreducible component of $S_m^i(T, R^j \fl \omX)$ is a translate of a subtorus of
$\Pic^0(T)$ by a point of finite order.
\end{corollary}

\begin{proof}
As in the proof of \theoremref{thm:direct_image} (in \parref{par:ChenJiang}), we let
$M \in \HM{T}{\dim X + j}$ be the summand with strict support $f(X)$ in the decomposition
by strict support of $\shH^j \fl \RR_X \decal{\dim X}$; then $M$ admits an integral
structure, and 
\[
	R^j \fl \omX \simeq F_{p(M)} \Mmod,
\]
where $p(M)$ again means the smallest integer such that $F_p \Mmod \neq 0$.
Since $M$ still admits an integral structure by \lemmaref{lem:integral-summand}, the
result in \corollaryref{cor:finite-order} shows that the sets
\[
	S_m^i(T, M) = \menge{\rho \in \Char(T)}%
		{\dim H^i(T, \ratM \tensor_{\RR} \CCrho) \geq m}
\]
are finite unions of translates of linear subvarieties by points of finite order. As
in the proof of \lemmaref{lem:M-regular}, the strictness of the complex computing the
hypercohomology of $(M \oplus M, J_M) \tensor_{\CC} \CCrho$ implies that
\[
	\dim H^i(T, \ratM \tensor_{\RR} \CCrho) = \sum_{p \in \ZZ} 
		\dim H^i \bigl( T, \gr_p^F \DR(\Mmod) \tensor_{\OT} L_{\rho} \bigr)
\]
for every unitary character $\rho \in \Char(T)$; here $L_{\rho} = \CCrho
\tensor_{\CC} \OT$. Note that $\gr_p^F \DR(\Mmod)$ is acyclic for $p \gg 0$, and so
the sum on the right-hand side is actually finite. Intersecting $S_m^i(T, M)$ with
the subgroup of unitary characters, we see that each~set
\[
	\Menge{L \in \Pic^0(T)}%
		{\sum_{p \in \ZZ} \dim H^i \bigl( T, \gr_p^F \DR(\Mmod) 
			\tensor_{\OT} L \bigr) \geq m}
\]
is a finite union of translates of subtori by points of finite order. By a standard
argument \cite[p.~312]{Arapura}, it follows that the same is true for each of the
summands; in other words, for each $p \in \ZZ$, the set
\[
	S_m^i \bigl( T, \gr_p^F \DR(\Mmod) \bigr) \subseteq \Pic^0(T)
\]
is itself a finite union of translates of subtori by points of finite order. Since
\[
	\gr_{p(M)}^F \DR(\Mmod) = \omega_T \tensor F_{p(M)} \Mmod \simeq R^j \fl \omX,
\]
we now obtain the assertion by specializing to $p = p(M)$.
\end{proof}

\begin{note}
Alternatively, one can deduce \corollaryref{cor:Wang} from Wang's theorem \cite{Wang}
about cohomology jump loci on compact K\"ahler manifolds, as follows. Wang shows that 
the sets $S_m^{p,q}(X) = \menge{L \in \Pic^0(X)}{\dim H^q(X, \OmX^p
\tensor L) \geq m}$ are finite unions of translates of subtori by points of finite
order; in particular, this is true for $\omX = \OmX^{\dim X}$. Takegoshi's results
about higher direct images of $\omX$ in \theoremref{takegoshi} imply the
$E_2$-degeneration of the spectral sequence
\[
	E_2^{i,j} = H^i \bigl( T, R^j \fl \omX \tensor L \bigr)
		\Longrightarrow H^{i+j}(X, \omX \tensor \fu L)
\]
for every $L \in \Pic^0(T)$, which means that
\[
	\dim H^q(X, \omX \tensor \fu L) 
		= \sum_{k+j=q} \dim H^k \bigl( T, R^j \fl \omX \tensor L \bigr).
\] 
The assertion now follows from Wang's theorem by the same argument as above.
\end{note}

\section{Applications}

\subsection{Bimeromorphic characterization of tori}

Our main application of generic vanishing for higher direct images of dualizing
sheaves is an extension of the
Chen-Hacon birational characterization of abelian varieties \cite{CH1} to the
K\"ahler case. 

\begin{theorem}\label{torus}
Let $X$ be a compact K\"ahler manifold with $P_1(X) = P_2 (X) = 1$ and $h^{1,0}(X) =
\dim X$. Then $X$ is bimeromorphic to a compact complex torus. 
\end{theorem}

Throughout this section, we take $X$ to be a compact K\"ahler manifold, and denote by $f
\colon X \to T$ its Albanese mapping; by assumption, we have 
\[
	\dim T = h^{1,0}(X) = \dim X.
\]
We use the following standard notation, analogous to that in \parref{par:GV-sheaves}:
\[
	S^i (X, \omX) = \menge{L \in \Pic^0(X)}{H^i (X, \omega_X \tensor L) \neq 0}
\]
To simplify things, we shall identify $\Pic^0(X)$ and $\Pic^0(T)$ in what
follows. We begin by recalling a few well-known results.

\begin{lemma}\label{isolated}
If $P_1(X) = P_2(X) = 1$, then there cannot be any positive-dimensional analytic
subvariety $Z \subseteq \Pic^0 (X)$ such that both $Z$ and $Z^{-1}$ are contained in $S^0
(X, \omega_X)$. In particular, the origin must be an isolated point in $S^0 (X, \omega_X)$.
\end{lemma} 
\begin{proof}
This result is due to Ein and Lazarsfeld \cite[Proposition~2.1]{EL}; they state it
only in the projective case, but their proof actually works without any changes on
arbitrary compact K\"ahler manifolds.
\end{proof} 

\begin{lemma}\label{surjective}
Assume that $S^0 (X, \omega_X)$ contains isolated points. Then the Albanese map of $X$ 
is surjective.
\end{lemma}
\begin{proof}
By \theoremref{thm:direct_image} (for $j = 0$), $\fl \omega_X$ is a GV-sheaf.
\propositionref{sliding}
shows that any isolated point in $S^0 (T, \fl \omega_X) = S^0 (X, \omega_X)$ also
belongs to $S^{\dim T}(T, \fl \omega_X)$; but this is only possible if the support of
$\fl \omX$ has dimension at least $\dim T$.
\end{proof}

To prove \theoremref{torus}, we follow the general strategy introduced in
\cite[\S4]{Pareschi}, which in turn is inspired by \cite{EL,CH2}. The crucial new
ingredient is of course \theoremref{thm:direct_image}, which had only been known in the
projective case. Even in the projective case however, the argument below is 
substantially cleaner than the existing proofs; this is due to \corollaryref{cor:gen-finite}.

\begin{proof}[Proof of \theoremref{torus}]
The Albanese map $f \colon X \to T$ is surjective by \lemmaref{isolated} and
\lemmaref{surjective}; since $h^{1,0}(X) = \dim X$, this means that $f$ is
generically finite.  To conclude the proof, we just have to argue that $f$ has degree
one; more precisely, we shall use \theoremref{thm:direct_image} to show that $\fl \omX
\simeq \OT$. 

As a first step in this direction, let us prove that $\dim S^0(T, \fl \omX) = 0$. If
\[
	S^0(T, \fl \omX) = S^0(X, \omX)
\]
had an irreducible component $Z$ of positive dimension, \corollaryref{cor:gen-finite}
would imply that $Z^{-1}$ is contained in $S^0(X, \omX)$ as well. As this would
contradict \lemmaref{isolated}, we conclude that $S^0(T, \fl \omX)$ is
zero-dimensional.

Now $\fl \omX$ is a GV-sheaf by \theoremref{thm:direct_image}, and so \propositionref{sliding}
shows that
\[
	S^0(T, \fl \omX) = S^{\dim T}(T, \fl \omX).
\]
Since $f$ is generically finite, \theoremref{takegoshi} implies that $R^j \fl \omX = 0$ for $j > 0$, which
gives
\[
	S^{\dim T}(T, \fl \omX) = S^{\dim T}(X, \omX) 
		= S^{\dim X}(X, \omX) = \{\OT\}.
\]
Putting everything together, we see that $S^0(T, \fl \omX) = \{\OT\}$.

We can now use the Chen-Jiang decomposition for $\fl \omX$ to get more information. 
The decomposition in \theoremref{thm:direct_image} (for $j = 0$) implies that
\[
	\{\OT\} = S^0(T, \fl \omX) = \bigcup_{k=1}^n L_k^{-1} \tensor \Pic^0(T_k),
\]
where we identify $\Pic^0(T_k)$ with its image in $\Pic^0(T)$. This equality forces
$\fl \omX$ to be a trivial bundle of rank $n$; but then
\[
	n = \dim H^{\dim T}(T, \fl \omX) = \dim H^{\dim X}(X, \omX) = 1,
\]
and so $\fl \omX \simeq \OT$.	The conclusion is that $f$ is generically
finite of degree one, and hence birational, as asserted by the theorem.
\end{proof}

\subsection{Connectedness of the fibers of the Albanese map}

As another application, one obtains the following analogue of an effective version of
Kawamata's theorem on the connectedness of the fibers of the Albanese map, proved by
Jiang \cite[Theorem 3.1]{jiang} in the projective setting. Note that the statement is
more general than \theoremref{torus}, but uses it in its proof.

\begin{theorem}\label{thm:jiang}
Let $X$ be a compact K\"ahler manifold with $P_1(X) = P_2 (X) = 1$. Then the Albanese map of $X$ is 
surjective, with connected fibers.
\end{theorem}
\begin{proof}
The proof goes entirely along the lines of \cite{jiang}. We only indicate the necessary modifications in the K\"ahler case.
We have already seen that the Albanese map $f \colon X \to T$ is surjective. Consider
its Stein factorization 
\[
\begin{tikzcd}[column sep=large]
X \dar{g} \drar{f} &  \\
Y \rar{h} &  T.
\end{tikzcd}
\]
Up to passing to a resolution of singularities and allowing $h$ to be generically
finite, we can assume that $Y$ is a compact complex manifold. Moreover, by
\cite[Th\'eor\`eme 3]{Varouchas}, after performing a further bimeromorphic
modification, we can assume that $Y$ is in fact compact K\"ahler.  This does not
change the hypothesis $P_1(X) = P_2(X) = 1$.

The goal is to show that $Y$ is bimeromorphic to a torus, which is enough to
conclude. If one could prove that $P_1(Y) = P_2 (Y) = 1$, then 
\theoremref{torus} would do the job. In fact, one can show precisely as in
\cite[Theorem 3.1]{jiang} that $H^0 (X, \omega_{X/Y}) \neq 0$, and consequently that
\[
	P_m (Y) \le P_m (X) \quad \text{for all $m \geq 1$.}
\]
The proof of this statement needs the degeneration of the Leray spectral sequence for
$g_* \omega_X$, which follows from  \theoremref{takegoshi}, and the fact that $f_*
\omega_X$ is a GV-sheaf, which follows from \theoremref{thm:direct_image}. Besides this, the
proof is purely Hodge-theoretic, and hence works equally well in the K\"ahler case.
\end{proof}

\subsection{Semi-positivity of higher direct images}
\label{par:semi-positivity}

In the projective case, GV-sheaves automatically come with positivity properties;
more precisely, on abelian varieties it was proved in \cite[Corollary~3.2]{Debarre}
that $M$-regular sheaves are ample, and in \cite[Theorem~ 4.1]{PP2} that GV-sheaves
are nef. Due to \theoremref{thm:Chen-Jiang} a stronger result in fact holds true,
for arbitrary graded quotients of Hodge modules on compact complex tori.

Recall that to a coherent sheaf $\shF$ on a compact complex manifold one can
associate the analytic space $\PP (\shF) = \PP \left( {\rm Sym}^\bullet \shF
\right)$, with a natural mapping to $X$ and a line bundle $\shO_{\PP(\shF)} (1)$. If
$X$ is projective, the sheaf $\shF$ is called \emph{ample} if the line bundle
$\shO_{\PP(\shF)} (1)$ is ample on $\PP(\shF)$.

\begin{corollary}\label{cor:MHM_positivity}
Let $M = (\Mmod, F_{\bullet} \Mmod, \ratM)$ be a polarizable real Hodge
module on a compact complex torus $T$. Then, for each $k \in \ZZ$, the coherent
$\OT$-module $\gr_k^F \Mmod$ admits a decomposition
\[
	\gr_k^F \Mmod \simeq \bigoplus_{j=1}^n 
		\bigl( \qu_j \shF_j \tensor_{\OT} L_j \bigr),
\]
where $q_j \colon T \to T_j$ is a quotient torus, $\shF_j$ is an
ample coherent $\shO_{T_j}$-module whose support $\Supp \shF_j$ is projective,
and $L_j \in \Pic^0(T)$.
\end{corollary}
\begin{proof}
By \theoremref{thm:Chen-Jiang} we have a decomposition as in the statement, where
each $\shF_j$ is an $M$-regular sheaf on the abelian variety generated by its
support.  But then \cite[Corollary~3.2]{Debarre} implies that each $\shF_j$ is ample.
\end{proof}

The ampleness part in \theoremref{thm:direct_image} is then a consequence of the
proof in \parref{par:ChenJiang} and the statement above. It implies that higher
direct images of canonical bundles have a strong semi-positivity property
(corresponding to semi-ampleness in the projective setting). Even the following very
special consequence seems to go beyond what can be said for arbitrary holomorphic
mappings of compact K\"ahler manifolds (see e.g. \cite{MT} and the references
therein). 

\begin{corollary} \label{cor:semi-pos}
Let $f: X \rightarrow T$ be a surjective holomorphic mapping from a compact K\"ahler
manifold to a complex torus. If $f$ is a submersion outside of a simple normal
crossings divisor on $T$, then each $R^i f_* \omega_X$ is locally free and admits a
smooth hermitian metric with semi-positive curvature (in the sense of Griffiths).
\end{corollary}

\begin{proof}
Note that if $f$ is surjective, then \theoremref{takegoshi} implies that $R^i f_*
\omega_X$ are all torsion free.  If one assumes in addition that $f$ is a submersion
outside of a simple normal crossings divisor on $T$, then they are locally free; see
\cite[Theorem~V]{Takegoshi}. Because of the decomposition in
\theoremref{thm:direct_image}, it is therefore enough to show that an M-regular
locally free sheaf on an abelian variety always admits a smooth hermitian metric with
semi-positive curvature. But this is an immediate consequence of the fact that
M-regular sheaves are continuously globally generated \cite[Proposition~2.19]{PP1}.
\end{proof}

The existence of a metric with semi-positive curvature on a vector bundle $E$ implies
that the line bundle $\shO_{\PP(E)}(1)$ is nef, but is in general known to be a
strictly stronger condition. \corollaryref{cor:semi-pos} suggests the following
question.

\begin{problem}
Let $T$ be a compact complex torus. Suppose that a locally free sheaf $\shE$ on $T$
admits a smooth hermitian metric with semi-positive curvature (in the sense of
Griffiths or Nakano). Does this imply the existence of a decomposition
\[
	\shE \simeq \bigoplus_{k=1}^n \bigl( \qu_k \shE_k \tensor L_k \bigr)
\]
as in \theoremref{thm:direct_image}, in which each locally free sheaf $\shE_k$
has a smooth hermitian metric with strictly positive curvature?
\end{problem}

\subsection{Leray filtration}

Let $f \colon X \to T$ be a holomorphic mapping from a compact K\"ahler manifold $X$
to a compact complex torus $T$.  We use \theoremref{thm:direct_image} to describe the Leray
filtration on the cohomology of $\omX$, induced by the Leray spectral sequence
associated to $f$.  Recall that, for each $k$, the Leray filtration on $H^k(X, \omX)$
is a decreasing filtration $L^{\bullet} H^k(X, \omX)$ with the property that
\[
	\gr_L^i H^k(X, \omX) = H^i \bigl( T, R^{k-i} \fl \omX \bigr).
\]
On the other hand, one can define a natural decreasing filtration $F^{\bullet} H^k
(X, \omega_X)$ induced by the action of $H^1(T, \shO_T)$, namely 
\[
	F^i H^k(X, \omX) = \Im \left( \bigwedge^i H^1(T, \shO_T) \tensor H^{k-i} (X, \omX) 
		\to H^{k}(X, \omX)\right).
\]
It is obvious that the image of the cup product mapping
\begin{equation} \label{eq:cup-product}
	H^1(T, \shO_T) \tensor L^i H^k(X, \omX) \to H^{k+1}(X, \omX)
\end{equation}
is contained in the subspace $L^{i+1} H^{k+1}(X, \omX)$. This implies that 
\[
	F^i H^k (X, \omega_X) \subseteq L^i H^k (X, \omega_X) 
		\quad \text{for all $i \in \ZZ$.}
\]
This inclusion is actually an equality, as shown by the following result.

\begin{theorem}\label{leray}
The image of the mapping in \eqref{eq:cup-product} is equal to $L^{i+1} H^{k+1}(X,
\omX)$. Consequently, the two filtrations $L^{\bullet} H^k(X, \omX)$ and $F^{\bullet}
H^k(X, \omX)$ coincide.
\end{theorem}

\begin{proof}
By \cite[Theorem~A]{LPS}, the graded module
\[
	Q_X^j = \bigoplus_{i=0}^{\dim T} H^i \bigl( T, R^j \fl \omX \bigr)
\]
over the exterior algebra on $H^1(T, \shO_T)$ is $0$-regular, hence generated in
degree $0$. (Since each $R^j \fl \omX$ is a GV-sheaf by \theoremref{thm:direct_image}, the
proof in \cite{LPS} carries over to the case where $X$ is a compact K\"ahler
manifold.) This means that the cup product mappings
\[
	\bigwedge^i H^1(T, \shO_T) \tensor H^0 \bigl( T, R^j \fl \omX \bigr)
		\to H^i \bigl( T, R^j \fl \omX \bigr)	
\]
are surjective for all $i$ and $j$, which in turn implies that the mappings
$$H^1(T, \shO_T) \otimes \gr_L^i H^k(X, \omX) \to \gr_L^{i+1} H^{k+1} (X, \omX)$$ 
are surjective for all $i$ and $k$. This implies the assertion by ascending induction.
\end{proof}

If we represent cohomology classes by smooth forms, Hodge conjugation and Serre
duality provide for each $k \geq 0$ a hermitian pairing
$$H^0 (X, \Omega_X^{n-k}) \times H^k (X, \omega_X) \rightarrow \CC, \,\,\,\,(\alpha, \beta) \mapsto 
\int_X \alpha \wedge \overline{\beta},$$
where $n = \dim X$. The Leray filtration on $H^k (X , \omega_X)$ therefore induces a filtration on 
$H^0 (X, \Omega_X^{n-k})$; concretely, with a numerical convention which again gives us a decreasing 
filtration with support in the range $0, \ldots, k$, we have
\[
	L^i H^0 (X, \Omega_X^{n-k}) = \menge{\alpha \in H^0(X, \OmX^{n-k})}%
		{\alpha \perp L^{k+1 - i} H^k (X, \omega_X)}.
\]
Using the description of the Leray filtration in \theoremref{leray}, and the elementary fact that 
$$\int_X \alpha \wedge \overline{\theta\wedge \beta} = \int_X \alpha \wedge \overline{\theta} \wedge 
\overline{\beta}$$
for all $\theta \in H^1 (X, \shO_X)$, we can easily deduce that $L^i H^0 (X, \Omega_X^{n-k})$ consists of those 
holomorphic $(n-k)$-forms whose wedge product with 
$$\bigwedge^{k+1-i} H^0 (X, \Omega_X^1)$$
vanishes. In other words, for all $j$ we have:

\begin{corollary}\label{cor:Leray_forms}
The induced Leray filtration on $H^0 (X, \Omega_X^j)$ is given by 
\[
	L^i H^0 (X, \Omega_X^j) = \Menge{\alpha \in H^0(X, \OmX^j)}%
		{\alpha \wedge \bigwedge^{n+1 - i - j} H^0 (X, \Omega_X^1) = 0}.
\]
\end{corollary}

\begin{remark}
It is precisely the fact that we do not know how to obtain this basic description of
the Leray filtration using standard Hodge theory that prevents us from giving a proof
of \theoremref{thm:direct_image} in the spirit of \cite{GL1}, 
and forces us to appeal to the theory of Hodge modules for the main results.
\end{remark}

\subsection*{Acknowledgements}

We thank Christopher Hacon, who first asked us about the behavior of higher direct
images in the K\"ahler setting some time ago, and with whom we have had numerous
fruitful discussions about generic vanishing over the years. We also thank Claude Sabbah
for advice about the definition of polarizable complex Hodge modules, and Jungkai
Chen, J\'anos Koll\'ar, Thomas Peternell, and Valentino Tosatti for useful discussions.

During the preparation of this paper CS has been partially supported by the NSF grant
DMS-1404947, MP by the NSF grant DMS-1405516, and GP by the MIUR PRIN project
``Geometry of Algebraic Varieties''.

%++++++++++++++++++++++++++++++++++

\bibliographystyle{amsalpha}
\bibliography{bibliography}

\end{document}